\documentclass[12pt,a4paper,oneside]{amsart}
\usepackage[utf8]{inputenc}
\usepackage{amsfonts, amsmath, amssymb, amsthm, amscd}
\usepackage{upgreek}
\usepackage{anysize}
\usepackage{mathtools}
\marginsize{2cm}{2cm}{1.5cm}{1.5cm}
\usepackage[english]{babel}
\usepackage{cmap}
\usepackage{enumerate}
\usepackage{xcolor}
\usepackage{hyperref}
\hypersetup{
	colorlinks   = true, %Colours links instead of ugly boxes
	urlcolor     = blue, %Colour for external hyperlinks
	linkcolor    = blue, %Colour of internal links
	citecolor    = red %Colour of citations
}
\newcommand{\Href}[2]{\hyperref[#2]{#1~\ref{#2}}}
\usepackage{anysize}
\marginsize{1.8cm}{1.8cm}{1.5cm}{1.5cm}
\numberwithin{equation}{section}
\newtheorem{thm}{Theorem}[section]
\newtheorem{prp}{Proposition}[section]
\newtheorem{lemma}{Lemma}[section]
\newtheorem{cor}{Corollary}[section]
\newtheorem{claim}{Claim}[section]

\newtheorem{rem}{Remark}[section]
\theoremstyle{definition}
\newtheorem{dfn}{Definition}[section]
\newtheorem{exl}[dfn]{Example}

\DeclareMathOperator{\dom}{dom}
\DeclareMathOperator{\supp}{supp}

\DeclareMathOperator{\pos}{\mathrm{pos}}
\DeclareMathOperator{\diag}{\mathrm{diag}}
\newcommand{\Red}{\Re^d}

\newcommand{\ball}[1]{\mathbf{B}^{#1}}
\newcommand{\norm}[1]{\left\|#1\right\|}
\newcommand{\enorm}[1]{\left|#1\right|}
\newcommand{\bd}[1]{\mathrm{bd}\left(#1\right)}

\newcommand{\iprod}[2]{\left\langle#1,#2\right\rangle}
\newcommand{\tr}[1]{\mathrm{trace}\left(#1\right)}

\def\N{{\mathbb N}}
\def\R{{\mathbb R}}

\renewcommand{\Re}{\mathbb{R}}

\def\ellips{\mathcal{E}}%
\def\ellipsd{\ellips^{d}}%

\newcommand{\st}{:}

\def\phi{\varphi}
\def\epsilon{\varepsilon}
\def\alpha{\upalpha}

\newcommand{\vol}[1]{\operatorname{vol}\nolimits_{#1}}%
\newcommand{\di}{\,\mathrm{d}}%integration d

\newcommand{\upthing}[1]{\overline{#1}}%
\newcommand{\sthing}[2][s]{\prescript{(#1)}{}{#2}}%
\newcommand{\selldense}[2][s]{\sthing[#1]{J}_{#2}}%
\newcommand{\tpdfs}{\texorpdfstring{$s$}{s}}%
\newcommand{\tpdfpsi}{\texorpdfstring{$\psi$}{psi}}%
\newcommand{\soelldense}[2][s]{\sthing[#1]{L}_{#2}}%

\newcommand{\psidense}[2][\psi]{{L_{#2}[#1]}}%
\newcommand{\psiell}[1][\upthing{E}]
{{\ell_{\psi, #1}}}%
\newcommand{\sell}[2][s]
{{\sthing[#1]\ell_{#2}}}%
\newcommand{\shf}[2][s]
{{\sthing[#1]h_{#2}}}%
\newcommand{\psiellf}[2][\upthing{E}]
{\psiell[#1]\!\left({#2}\right)}%
\newcommand{\funclass}[2][d]{\ellips^{#1}\! \left[ #2\right]}

\newcommand{\noshow}[1]{}%For commenting out larger parts

\newcommand{\irat}{\operatorname{I.rat}}%
\newcommand{\sintrat}[2][s]{{
\sthing[#1]{\irat}}{#2}}%
\newcommand{\oirat}{\operatorname{I.or}}%
\newcommand{\sointrat}[2][s]{{
\sthing[#1]{\oirat}}{#2}}%
\newcommand{\ovolbs}[1][s]{\sthing[#1]{\lambda}_{d}}%
\newcommand{\psiovolb}[1][\psi]{{\lambda_{d}[#1]}}%

\newcommand{\id}{\mathrm{Id}}
\def\legendre{\mathcal{L}}%
\newcommand{\loglego}[1]{{#1}^{\circ}}

\def\loginfcop{\star}%
\newcommand{\loginfconv}[2]{#1 \loginfcop #2}%

  \usepackage{setspace}

\usepackage{color}
\newcommand{\marrow}{\marginpar{$\longleftarrow$}}
   
\newcommand{\grisha}[1]{\textcolor{cyan}{\marrow #1}}

\title{Functional L\"owner Ellipsoids}
\author{Grigory Ivanov\address{Grigory Ivanov: 
Institute of Science and Technology Austria (IST Austria), 
Kleusteneuburg, Austria.  Moscow Inst. of Physics and Technology, 
Moscow, Russia}
\email{grimivanov@gmail.com}
\and
Igor Tsiutsiurupa
\address{Igor Tsiutsiurupa:  Moscow Inst. of Physics and Technology, 
Moscow, Russia}
\email{igor.tsutsurupa97@gmail.com}
}

\begin{document}
\begin{abstract}
	We extend the notion of the smallest volume ellipsoid containing a convex body in~$\mathbb{R}^{d}$
	to the setting of logarithmically concave functions. 
	We consider a vast class of logarithmically concave functions whose superlevel sets are concentric ellipsoids.
	For a fixed function from this class, we consider the set of all its ``affine'' positions.
	For any log-concave function $f$ on $\mathbb{R}^{d},$ 
	we consider functions belonging to this set of ``affine'' positions,
	and find the one with the smallest integral under the condition that it is pointwise greater than or equal to $f.$ 
	We study the properties of existence and uniqueness of the solution to this problem. 
	For any 
$s \in [0,\infty),$
	 we consider the construction dual to the recently defined John $s$-function~\cite{ivanov2020functional}.
	We prove that such a construction determines a unique function and call it the \emph{L\"owner $s$-function} of $f.$ 
	We study the L\"owner $s$-functions as $s$ tends to zero and to infinity.
	Finally, extending the notion of the outer volume ratio,
	we define the outer integral ratio of a log-concave function and give an asymptotically tight bound on it.
\end{abstract}
\maketitle

\section{Introduction}
In \cite{John}, Fritz John proved that each convex body $K$ in {$\R^d$} contains a unique ellipsoid of maximal volume, now called the \emph{John ellipsoid} of $K$. {A closely related object is the (unique) ellipsoid of minimal volume containing $K$, called the \emph{L\"owner ellipsoid} of $K$. The connection is via polar duality: If $K$ is a convex body in $\R^d$ such that the origin is the center of its John ellipsoid then the polar of the John ellipsoid of $K$ is the L\"owner ellipsoid of the polar set of $K.$
%We recall that the \textit{John ellipsoid} of a convex body $K$ is the ellipsoid of maximal volume contained within $K$
%and the {\it L\"owner ellipsoid} of $K$ is the ellipsoid of minimal volume containing~$K.$
%In~\cite{John}, F. John proves that each convex body in $\R^d$ contains a unique ellipsoid of maximal volume
%and characterizes all convex bodies $K$ such that the ellipsoid of maximal volume in $K$ is the Euclidean unit ball.  
The John and L{\"o}wner ellipsoids  are  the cornerstones of the modern convex analysis.
These objects appear in different areas of mathematics, computational mathematics and physics
(see~\cite{ball2001convex, henk2012lowner, aubrun2017alice}).
An enormous number of problems has been solved using properties of these two deeply connected objects.

On the other hand,
the idea of extending the results of convex analysis to the more general setting of log-concave functions
has been investigated for the last few decades
(we refer to~\cite{brazitikos2014geometry, aubrun2017alice, artstein2015asymptotic}).
To the best of our knowledge,
D.~Alonso-Guti{\'e}rrez, B.G.~Merino, C.H.~Jim{\'e}nez,
  and R.~Villa~\cite{alonso2018john} were the first
who extended the notion of the John ellipsoid to the setting of integrable log-concave functions, and 
the first extension of the notion of L\"owner ellipsoid
 was introduced by  B.~Li, C.~Sh\"utt, and E.~Werner in~\cite{li2019loewner}. Let us explain their approaches.
%\emph{The John function in the sense of}~\cite{alonso2018john} is defined as follows. 

%We say that a function $f_1\colon \R^d \to \R$ is \emph{below} a function 
%$f_2\colon \R^d \to \R$ and write $f_1\leq f_2$
% if $f_1$ is pointwise less than or equal to $f_2$, that is, 
% $f_1 (x) \leq f_2 (x)$ for all $x \in \R^d.$
%The characteristic function of a set $K\subset\R^d$ is denoted by $\chi_K$.

%Let $f\colon \R^d \to [0,\infty)$ be an upper semi-continuous log-concave  function of finite positive integral.
% Consider the following class of functions
% \[
% 	\mathcal{J}^d = 
% 		\left\{
% 			\alpha \chi_{E} \st E\subset\R^d  \;\text{is an ellipsoid}, \ \alpha > 0
% 		\right\},
% \]
% and the following problem:
% \[
% 	\max\limits_{h \in \mathcal{J}^d }
% 		\int_{\R^d} h
% 		\quad \text{subject to} \quad
% 		h \leq f.
% \]
Consider the class of scaled characteristic functions of ellipsoids in $\R^d,$ that is, the functions of the form $\alpha \chi_E,$
where $\alpha > 0$ and $\chi_E$ is the characteristic function of an ellipsoid $E$ in $\R^d.$ 
The authors of~\cite{alonso2018john} showed that, under a mild condition on a log-concave function $f \colon \R^d \to [0, \infty)$, there is a unique function, which we refer to as the \emph{John function of} $f$ \emph{in the sense of}~\cite{alonso2018john}, of this class of maximal integral that is pointwise less than or equal to $f.$ 

The authors of~\cite{li2019loewner} study the dual problem:
They consider the class of functions  of the form 
$ \alpha e^{-\enorm{\mathcal{A}x}}$, where $\alpha > 0$ and
 $\mathcal{A}$ is a nonsingular
affine transformation, and show that there is a unique function 
of minimal integral that is pointwise greater than or equal to a given upper semi-continuous log-concave function $f \colon \R^d \to [0, \infty)$ of finite positive integral. We call this unique function of minimal integral the \emph{L\"owner function of $f$ in the sense of}~\cite{li2019loewner}.

As in the setting of convex sets, the connection between these two optimization problems is via polar duality (see the definitions in Subsection~\ref{subsec:duality_def}). 
The \emph{polar} (or \emph{log-conjugate}) function of the characteristic function of the Euclidean unit ball is $e^{-\enorm{x}}.$ Therefore, the classes of functions considered in \cite{alonso2018john} and \cite{li2019loewner} consist of translates of functions
that are polar to each other.

%Consider the following optimization problem.
%Among the functions of the form $\ell = \alpha e^{-|A (x - a)|}$, where $\alpha > 0$, $A \in \operatorname{GL}(d)$ and $a \in \R^d$,
%satisfying constraint $f \leq \ell$,
%find a finction of a minimal integral.
%As shown in~\cite{li2019loewner},
%there is a unique solution to this problem,
%which we call the L\"owner function of $f$ in the sense of~\cite{li2019loewner}.
%Since the log-conjugate function (see the definition in Subsection~\ref{subsec:duality_def})
%of the function $\chi_{\ball{d}}$
%is $e^{-\enorm{x}},$
%the classes of functions of the form $h=\alpha\chi_{A\ball{d}+a}$ and $\ell=\alpha e^{-|A (x - a)|}$
%consist 
%Therefore, the problems studied in~\cite{alonso2018john} and in~\cite{li2019loewner}
%may be considered as dual in some sense.

Recently, a more general approach to the definition of John function 
has been considered in \cite{ivanov2020functional},
where the \emph{John $s$-functions} of an integrable log-concave function
are constructed for all $s \in (0, \infty).$
We define and discuss these functions in Subsection~\ref{subsec:john_s-func}.
It was shown~\cite[Theorem~7.3]{ivanov2020functional}
that the John function in the sense of~\cite{alonso2018john}
is the limit (in a rather strong sense)
of the John $s$-function as $s \to 0.$
Apart of this limit result, 
a characterization of the John $s$-function
similar to the one given by John in his fundamental theorem
is given in~\cite[Theorem~5.1]{ivanov2020functional}. 

Combining ideas of~\cite{li2019loewner} and~\cite{ivanov2020functional},
we consider the ``dual'' problem and define the L\"owner $s$-function below. 
In hindsight, we understood that it is easier to see the main ideas of our proofs and geometric construction
considering a more abstract approach than the one in~\cite{ivanov2020functional}.

%
%Since we have to give some preliminary definitions in the introduction to the paper, we ask the reader to accept out sincerest apologies. 
We formalize our approach in a functional way. 
The classical L\"owner ellipsoid can be introduced as follows. We consider the class of all nonsingular
affine images (we may call them \emph{positions}) of the unit ball $\ball{d}.$ The L\"owner ellipsoid of a convex body $K$ is the (unique) largest volume element of this class containing $K$.

For a function $\psi\colon [0, \infty) \to \R \cup \{+\infty\},$ 
we consider the following class of functions
\[
	\funclass{\psi} = 
		\left\{
		\alpha e^{- \psi(\enorm{A(x -a)})} \st A  \in \operatorname{GL}(d), \ \alpha > 0, \; a \in \R^d 
		\right\}.
\]
One may consider the functional class 
$\funclass{\psi}$ as the class of ``affine'' positions  of the function $e^{- \psi(\enorm{x})},$ $x \in \R^d.$ 
Now we can say that the classes of ``affine'' positions of the characteristic function of the unit ball  and the function
$ x \mapsto e^{-\enorm{x}},$ $x  \in \R^d,$ were considered in~\cite{alonso2018john} and~\cite{li2019loewner}, respectively.

We require for functions of $\funclass{\psi}$ to be  reasonably good log-concave functions.
We say that $\psi\colon [0, \infty) \to \R \cup \{+\infty\}$ is an \emph{admissible  function}
if the function $t \mapsto e^{-\psi(\enorm{t})},$ $t \in \R,$ 
is an upper semi-continuous log-concave  function of finite positive integral.
We discuss the properties of admissible functions in Subsection~\ref{subsec:adm_convex_func} and consider only classes 
$\funclass{\psi}$ with $\psi$ being admissible.

We will say that a function $f_1: \R^d \to \R$  is \emph{below} a function 
$f_2: \R^d \to \R$, {denoted} %and denote it as 
$f_1\leq f_2$,
 if $f_1$ is pointwise less than or equal to $f_2$, that is, 
 $f_1 (x) \leq f_2 (x)$ for all $x \in \R^d.$

For an admissible function $\psi \colon [0, \infty) \to \R \cup \{+\infty\}$ and an upper semi-continuous log-concave  function $f\colon \R^d \to [0, \infty)$ of finite positive integral,
we study the following optimization problem:
\begin{equation}
\label{eq:psi_problem}
	\min\limits_{\ell \in \funclass{\psi} }
		\int_{\R^d} \ell 
		\quad \text{subject to} \quad
		f \leq \ell.
\end{equation}

% We describe the class of admissible convex functions in Subsection~\ref{subsec:adm_convex_func}.
\subsection{The main results}
\label{subsec:main_results}
First, we study the uniqueness of the solution to~\eqref{eq:psi_problem}.
\begin{thm}\label{thm:psi_lowner}
	Let  $\psi\colon [0, \infty) \to \R \cup \{+\infty\}$
	be an admissible function
	and let $\psi$ satisfy one of the following conditions:
	\begin{enumerate}
		\item $\psi$ is strictly increasing and takes only finite values;
		\item the effective domain of $\psi$ is bounded.
	\end{enumerate}
	Further, let $f \colon \R^d \to [0, \infty)$ be  an upper semi-continuous log-concave  function  of finite positive integral.
	If there exists $\ell \in \funclass{\psi}$ such that $f \leq \ell,$ 
	then there exists a unique solution to problem~\eqref{eq:psi_problem}.
	Moreover, the supremum of a solution to this problem  is at most  $e^d$ the supremum of~$f.$ 
\end{thm}

% In Subsection~\ref{subsec:chimeras},
% we show that there is no uniqueness of the solution in general
% if $\psi$ does not satisfy both conditions in Theorem~\ref{thm:psi_lowner}.
% We discuss the properties of existence of the solution in Section~\ref{sec:classes}.
% We note here that we use two different geometric ideas
% when we consider admissible functions satisfying different conditions in Theorem~\ref{thm:psi_lowner}. 

The question of existence of $\ell \in \funclass{\psi}$ 
satisfying relation $f \leq \ell$ is not hard. 
As shown in Lemma~\ref{lem:psi_lin_growth_bound},
it suffices for $\psi$ to be of linear growth at infinity.

Next, we study the dual problem to that  considered in \cite{ivanov2020functional}.
For any $s \in [0, \infty) \cup \{\infty\}$, we define $\psi_s \colon [0, \infty) \to [0, \infty)$ by 
\begin{equation}\label{eq:ellipsoid-function_correspondence}
		\psi_s (t)=
			\begin{cases}
				t, 	&\quad s = 0\\
				\cfrac{s}{2}
					\left[ 
						\sqrt{1+4\left(\cfrac{t}{s}\right)^2} -
						\ln \!\left(\cfrac{1+\sqrt{1+4\left(\frac{t}{s}\right)^2}}{2}\right)  - 1
					\right],
				&\quad s \in (0, \infty) \\
				t^2, &\quad s = \infty
			\end{cases}.
	\end{equation}
	As it will be shown in Subsection~\ref{subsec:john_s-func},
the function $\psi_s$ is an admissible function for an arbitrary $s \in [0, \infty) \cup \{\infty\}.$
One sees that $\funclass{\psi_s}$ with $s = 0$ coincides with the class of functions considered in \cite{li2019loewner}, and that the L\"owner function in the sense of~\cite{li2019loewner}
is  a solution to problem \eqref{eq:psi_problem} with $\psi = \psi_0.$  The class $\funclass{\psi_{\infty}}$  consists of Gaussian densities. The cumbersome definition of $\psi_s$ in the case 
$s \in (0, \infty)$ is caused by polar duality.
 In Subsection~\ref{subsec:john_s-func}, we will show that the problem \eqref{eq:psi_problem} with $\psi = \psi_s$ and $s \in  (0, \infty)$ is dual to the problem of finding the John $s$-function considered in \cite{ivanov2020functional}. We study  problem \eqref{eq:psi_problem} with $\psi= \psi_s$ in the last four sections of the paper.

As a simple consequence of Theorem~\ref{thm:psi_lowner}, we obtain the following result.
\begin{thm}\label{thm:s_lowner}
	Fix $s \in [0, \infty)$ and let $f\colon\Red\to[0,\infty)$
	be an upper semi-continuous log-concave  function of finite positive integral.
	Then there exists a unique solution to problem
	\begin{equation}\label{eq:s_problem}
		\min\limits_{\ell \in \funclass{\psi_s} }
			\int_{\R^d} \ell
			\quad \text{subject to} \quad
			f \leq \ell.
	\end{equation}
\end{thm}
We will refer to the solution to problem~\eqref{eq:s_problem} for a fixed $s \in [0, \infty)$ 
as the \emph{L\"owner $s$-function} of~$f$,
and denote it by~$\soelldense{f}.$
Note that the L\"owner function in the sense of~\cite{li2019loewner}
is precisely~$\soelldense[0]{f}.$

As in the case of John $s$-functions,
we get the following limit result as $s$ tends to zero.
\begin{thm}\label{thm:zero_limit}
	Let $f\colon \R^d \to [0, \infty)$ be a an upper semi-continuous log-concave  function of finite positive integral.
	Then the L\"owner $s$-functions of $f$ converge uniformly {on $\R^d$} to the L\"owner function of $f$ in the sense of \cite{li2019loewner}  as 
	$s \to +0 .$ 
\end{thm}
In other words, $\soelldense{f} \to \soelldense[0]{f}$ uniformly on $\R^d.$ That is, the L\"owner $s$-functions of $f$
can be considered as a reasonable extension of the L\"owner function in the sense of~\cite{li2019loewner}.

Probably, the most striking difference between the John $s$-functions and the L\"owner $s$-functions
appears in the case $s = \infty.$ 
We prove the following statement.
\begin{thm}\label{thm:lowner_infinity_existence}
	Let $f\colon\Red\to[0,\infty)$ be a an upper semi-continuous log-concave  function of finite positive integral.
	Then the following assertions are equivalent:
	\begin{enumerate}
		\item\label{ass:thm:lowner_infinity_liminf} 
			\[
				\liminf\limits_{s \to \infty} \int_{\R^d} \soelldense{f}  < \infty;
			\]
		\item\label{ass:thm:lowner_infinity_exists_any}
			There exists a Gaussian density $G$ such that $f \leq G$;
		\item\label{ass:thm:lowner_infinity_exists_unique}
			There exists a unique solution to problem
			\begin{equation}\label{eq:s_infinity_problem}
				\min\limits_{\ell \in \funclass{\psi_\infty}}
					\int_{\R^d} \ell 
					\quad \text{subject to}\quad
					f \leq \ell.
			\end{equation}
		Also, the L\"owner $s$-functions of $f$ converge uniformly {on $\R^d$}
		to the solution to problem~\eqref{eq:s_infinity_problem}
		as $s \to \infty.$
	\end{enumerate}
\end{thm}
Surprisingly enough, it was shown in~\cite{ivanov2020functional}
that the Gaussian density of maximal integral below a given upper semi-continuous log-concave  function  of positive integral
is not necessarily unique.

Let $K$ be a convex body in $\R^d$ and $L_K$ be its L\"owner ellipsoid.
The ratio  
\[
	\left( 
		\frac{\vol{d} L_K}{\vol{d} K}
	\right)^{1/d} 
\]
is called the \emph{outer volume ratio} of $K.$
Using the reverse Brascamp--Lieb inequality, 
F. Barthe~\cite{barthe1998reverse} showed that for any positive integer $d$ and any convex body 
$K \in \R^d,$ the outer volume ratio of $K$
is at most $\Theta \sqrt{d}$
for some positive constant $\Theta.$
We extend this result to the setting of log-concave functions.
\begin{thm}\label{thm:s-outer_volume_ratio} 
	Fix $s \in [0, \infty).$
	There exists $\Theta_s$
	such that for any positive integer $d$ and an upper semi-continuous log-concave  function $f\colon \Red \to [0,\infty)$ of finite integral,
	the following inequality holds:
	\[
		\left( 
		\frac{\int_{\Red} \soelldense{f} }{\int_{\Red}  f}
		\right)^{1/d} \leq  
		 \Theta_s  \sqrt{d}.
	\] 
\end{thm}

\subsection{Organization of the paper.}

In Section~\ref{sec:notation},
we introduce the basic definitions and all needed terminology.
% We also give simple properties of admissible functions
% and describe the class of admissible functions in Subsection~\ref{subsec:adm_convex_func}.

In Section~\ref{sec:classes},
we discuss the properties of existence of a solution to problem~\eqref{eq:psi_problem} in a broad {sense}.
To establish the existence, we show in Lemma~\ref{lem:psi_boundedness}
that the set of parameters of problem~\eqref{eq:psi_problem} is a bounded set in a finite-dimensional linear space.
This yields a compactness result,
{namely,} the convergence of minimizing sequences (see Corollary~\ref{cor:existence}).
We also explain one of {our tools},
\emph{interpolation} between the functions of the class~$\funclass{\psi}.$
%,
%where $\psi$ is an admissible function (see Lemma~\ref{lem:psi_containment}).
We use this {tool}  to show  that 
all the solutions to problem~\eqref{eq:psi_problem} are 
{necessarily} translates of each other (see  Theorem~\ref{thm:uniqueness_upto_translation}).

In Section~\ref{sec:uniqueness}, {we investigate}
the question of uniqueness of a solution to problem~\eqref{eq:psi_problem}.
Given two functions $\ell_1,\ell_2\in\funclass{\psi}$ that are translates of each other,
we construct {a special function $\ell\in\funclass{\psi}$ that satisfies two conditions:}
First, $\ell_1 \leq \ell$ and $\ell_2 \leq \ell;$ second,  
$\ell$ is  of smaller integral than that of the functions $\ell_1$ or $\ell_2$
(see Lemma~\ref{lem:psi_sausage_lemma_st_increasing} and Lemma~\ref{lem:psi_sausage_lemma_bounded_domain}).
We note here
that we use two different geometric ideas
when the function $\psi$ satisfies different conditions in Theorem~\ref{thm:psi_lowner}.
Using this construction, we immediately get the
uniqueness result (see Theorem~\ref{thm:lowner_uniqueness_general}).
We also show in Subsection~\ref{subsec:chimeras}
that there is no uniqueness of a solution 
if $\psi$ does not satisfy both  conditions in Theorem~\ref{thm:psi_lowner}.

In Section~\ref{sec:existence},
{we show in Lemma \ref{lem:psi_lin_growth_bound}} that   a solution to problem~\eqref{eq:psi_problem} exists if $\psi$ is of linear growth at infinity.
In Lemma~\ref{lem:lowner_height_bounded},
we obtain the upper bound on the {supremum} of a solution
in terms of the supremum of $f$.
Summarizing the results,
we {complete the proof of} Theorem~\ref{thm:psi_lowner}.

In Section~\ref{sec:s_lowner},
we restrict ourselves to the function $\psi_s$, $s\in[0,+\infty)$.
We recall the definition of John $s$-functions introduced in~\cite{ivanov2020functional}.
Using the Legendre transform,
we obtain several simple properties of the function $\psi_s$ and
show that problem~\eqref{eq:psi_problem} and the problem considered in~\cite{ivanov2020functional} are dual.
This explains our motivation for introducing the function $\psi_s$.
We prove Theorem~\ref{thm:s_lowner} and discuss {the} properties of L\"owner $s$-functions.
Next, we consider the limit case as $s\to 0$ and prove Theorem~\ref{thm:zero_limit}.

In Section~\ref{sec:gaussians},
we prove Theorem~\ref{thm:lowner_infinity_existence}.
This shows that the choice of $\psi_{\infty}$ is indeed reasonable.

In Section~\ref{sec:outer_integral_ratio},
we prove Theorem~\ref{thm:s-outer_volume_ratio} and discuss its interrelation with the notions of volume ratio and outer volume ratio. 
% Therefore, the notion of outer $s$-integral ratio of
% a log-concave function $f$, introduced . we generalize the notion of outer volume ratio of a convex body.

Finally, in Section~\ref{sec:duality},
the duality between John $s$-function introduced in~\cite{ivanov2020functional} and our L\"owner $s$-function is discussed.

\section{Notation, Basic Terminology}\label{sec:notation}

\subsection{Matrices}

We will use $\prec$
to denote the standard partial order on the cone of positive semi-definite matrices,
that is,
we write $A \prec B$ if $B - A$ is positive-definite.
We recall the additive and the multiplicative form of \emph{Minkowski's determinant  inequality}.
Let $A$ and $B$ be positive-definite matrices of order $d$.
Then, for any $\lambda \in (0,1),$
	\begin{equation}\label{eq:minkowski_det_ineq}
		(\det( \lambda A + (1 - \lambda) B))^{1/d}
			\geq
			\lambda (\det A)^{1/d}
			+ (1 - \lambda) (\det B)^{1/d}
	\end{equation}
with equality if and only if $A = cB$ for some $c > 0$,
and 
	\begin{equation}\label{eq:minkowski_det_multipl_ineq}
		\det(\lambda A + (1 - \lambda)B)
			\geq
				(\det A)^{\lambda} \cdot (\det B)^{1 -\lambda}
	\end{equation}
with equality if and only if $A = B.$

\subsection{Functions. Log-concave functions}

A function $\psi\colon\Red\to\Re\cup\{+\infty\}$ is called \emph{convex} if 
$\psi((1-\lambda)x+\lambda y)\leq (1-\lambda)\psi(x)+\lambda\psi(y)$ for every 
$x,y\in\Red$ and $\lambda\in[0,1]$. 
A function $f$ on $\Red$ is \emph{logarithmically concave} (or, 
\emph{log-concave} for short) if $f=e^{-\psi}$ for a convex function $\psi$ on 
$\Red$. 
We say that a log-concave function $f$ on $\Red$ is a \emph{proper 
log-concave function} if $f$ is upper semi-continuous and has finite positive 
integral.

Clearly, if $f$ and $g$ are log-concave functions, then
	\begin{equation}\label{eq:psi}
		f \leq g
		\quad \Longleftrightarrow \quad
		-\log g \leq -\log f .
	\end{equation}

For a function $f\colon\Red\to\Re$ and a scalar $\alpha\in\Re$,
we denote the \emph{superlevel set} of $f$ by 
\[
	[f \geq \alpha] = \{x \in \Red \st f(x) \geq \alpha\}.
\]	
	
The $L_{\infty}$ norm of a function $f$ is denoted as $\norm{f}.$
Recall that the \emph{effective domain} of a convex function $\phi\colon\R^d\to\R\cup\{+\infty\}$
is the set $\dom \phi=\{x \in \R^d \st \phi(x) < +\infty\}.$

\subsection{Concept of duality}\label{subsec:duality_def}

Recall the definition of the classical \emph{convex conjugate} transform
(or \emph{Legendre transform})
$\legendre$ defined for functions $\phi: \R^d \to \R\cup \{-\infty, \infty\}$ by
\[
	(\legendre{\phi})(y) = \sup\limits_{x \in \R^d} \{\iprod{x}{y} - \phi(x)\}.
\]
The reasonable extension (justified in~\cite{artstein2007characterization, artstein2008concept, artstein2009concept})
of this notion to the setting of log-concave functions is the following.
Let $f = e^{-\psi} : \Red \to [0, \infty],$ 
then its \emph{log-conjugate} (or \emph{polar}) function is defined by
\[
	\loglego{f}(y) = e^{- (\legendre \psi)(y)} = 
	\inf\limits_{x \in \R^d} \frac{e^{-\iprod{x}{y}}}{f(x)}.
\]
Clearly, the log-conjugate function of any function is log-concave. 
It is known~\cite{artstein2004santalo} that the log-conjugate function of a proper log-concave function
is a proper log-concave function.
Also, if $f$ and $g$ are log-concave functions, then
\begin{equation}\label{eq:polarity_reverse}
	f \leq g
	\quad \Longleftrightarrow \quad
	\loglego{g} \leq \loglego{f}.
\end{equation}

The following result is proven in~\cite[Lemma 3.2]{artstein2004santalo}.
\begin{lemma}\label{lem:convergence_dual}
	Let $f, \{f_i\}_1^{\infty} : \R^d \to [0, \infty)$ be log-concave functions
	such that $f_n \to f$ on a dense set $A \subset \R^d.$
	Then
	\begin{enumerate}
		\item\label{ass:lem_conv_dual_convergence_integrals} 
			$\int_{\R^d} f_i \to \int_{\R^d} f$,
		\item\label{ass:lem_conv_dual_convergence_itself}
			$\loglego{(f_n)} \to \loglego{f}$ locally uniformly on the interior of the support of $\loglego{f}.$
	\end{enumerate}
\end{lemma}

\subsection{Ellipsoids}

We denote the Euclidean unit ball in $\Re^n$ by $\ball{n}$, where $n$ will 
mostly be~$d$ or~$d+1.$
For a matrix $A \in \R^{d \times d}$ and a number $\alpha,$  $A\oplus\alpha$ denotes $(d+1)\times (d+1)$ matrix 
\[
	A \oplus \alpha
	=\left(
		\begin{array}{cc}
			A & 0\\
			0 & \alpha
		\end{array}
	\right).
\]
 We introduce the convex cone 
	\begin{equation*}%\label{eq:Edef}
		\ellipsd = \left\{(A\oplus\alpha, a) \st A\in\Re^{d\times d}
		\text{ is positive-definite}, \alpha>0, a \in \R^d \right\}.
	\end{equation*}
There is a one-to-one correspondence between $\ellipsd$
and the class of $(d+1)$-dimensional ellipsoids in $\R^{d+1}$ symmetric with respect to $\R^d$
(for example, $(A\oplus\alpha,a) \mapsto (A\oplus\alpha) \ball{d+1} + a$).
We will refer to the elements of $\ellipsd$ as  $d$-ellipsoids.

\subsection{Admissible convex functions}\label{subsec:adm_convex_func}

Clearly,  $\psi\colon [0, \infty) \to \R \cup \{+\infty\}$ is an
{admissible  function} if and only if $\psi$ has the following properties:
\begin{itemize}
	\item $\psi$ is convex with minimum at $0$
		(otherwise, $e^{-\psi(|t|)}$ is not log-concave);
	\item $\psi$ is lower semi-continuous
		(otherwise, $e^{-\psi(|t|)}$ is not upper semi-continuous);
	\item $\lim\limits_{t \to +\infty} \psi(t)= + \infty$
		(otherwise, the integral of $e^{-\psi(|t|)}$ equals $+ \infty$);
	\item $\dom \psi$ has positive measure
		(otherwise, the integral of $e^{-\psi(|t|)}$ equals zero).
\end{itemize}  
Recall that a convex function
is continuous (even locally Lipschitz) on the interior of the effective domain
(see~\cite[Proposition 2.2.6]{clarke1990optimization}).

We say that an admissible function $\psi$ is \emph{of linear growth}
if inequality $\psi(t) \leq c t$ holds for some constant $c$ and all $t \in [0,\infty).$
By convexity, $\psi$ is an admissible function of linear growth
if and only if the limit
	$\lim\limits_{t \to \infty} \frac{\psi(t)}{t}$ 
exists and is finite.
For example, $\psi_s$
is an admissible function of linear growth for any $s\in [0, \infty),$
as it is shown in~Subsection~\ref{subsec:john_s-func}

Let $\psi\colon [0, \infty) \to \R \cup \{+\infty\}$ be a convex function.  
Since $e^{-\psi(t) + c}=e^c e^{-\psi(t)},$
we have that $\funclass{\psi} = \funclass{\psi + c}$ for any constant  $c \in \R.$
We will use this observation and assume sometimes
that an admissible function is zero at zero.

\subsection{Classes of ellipsoidal functions}

We say that the functions of $\funclass{\psi}$ are 
\emph{$\psi$-ellipsoidal} functions.  
If $\psi$ is an admissible function, then all the functions of $\funclass{\psi}$
are proper log-concave functions.
 
Given $(A\oplus\alpha,a) \in \ellipsd,$
we  say that the $\psi$-ellipsoidal function on $\R^d$ defined by 
$x \mapsto \alpha \cdot e^{- \psi(\enorm{A(x -a)})}$
is \emph{represented} by $(A\oplus\alpha,a)$.
We use $\psiell$ to denote the $\psi$-ellipsoidal function
represented by $\upthing{E} = (A\oplus\alpha,a) \in \ellipsd.$ 
By definition, we have
	\begin{equation}\label{eq:psi_st_form}
		\psiellf{x}=\alpha 	\cdot \psiellf[\ball{d+1}]{A(x-a)}.
	\end{equation}
This simple identity plays a crucial role in  our proofs.

By the polar decomposition theorem, any $\psi$-ellipsoidal function is represented by 
a unique element of $\ellipsd.$ 
Thus,
\[
	\funclass{\psi} = 
		\left\{
			\alpha e^{- \psi(\enorm{A(x-a)})}  \st (A \oplus \alpha, a) \in \ellipsd
		\right\} = 
		\left\{
			\psiell  \st (A \oplus \alpha, a) \in \ellipsd
		\right\}.
\]
The number $\alpha$ is called the \emph{height} of the $\psi$-ellipsoidal function $\psiell[(A \oplus \alpha, a)].$ 

We use $\psiovolb$ to denote the integral of $\psi$-ellipsoidal function
represented by the unit Euclidean ball $\ball{d+1}$,
that is,
	$$\psiovolb=\int_{\R^d} \psiell[\ball{d+1}].$$
Clearly, $\psiovolb$ is a positive number.
Also, for an arbitrary $\upthing{E}=(A\oplus\alpha,a)\in\mathcal{E},$ we have that
	\begin{equation}\label{eq:psi_outer_volume}
		\int_{\R^d} \psiell =	 \frac{\alpha}{\det A} \cdot \psiovolb.
	\end{equation}
%\TODO{GENERAL STATEMENT}
%Clearly,
%$
% \soelldense{f}
%$
%is the L\"owner $s$-function of $f$
% if and only if 
%${\gamma  \cdot \soelldense{f}}$
%is the L\"owner $s$-function of  $\gamma f.$
%
%Similarly, for any affine map 
%$\mathcal{A}: \Red \to \Red,$
%$
% \soelldense{f}
%$
%is the L\"owner $s$-function  of $f$
% if and only if
%${\soelldense{f}}\circ \mathcal{A}$
%is the L\"owner $s$-function  of $f \circ \mathcal{A}.$

\subsection{Classes of \texorpdfstring{$\psi_s$}{psis}-ellipsoidal functions}
We will use  the following notations related to $\psi_s$
to stress the fact that we restrict ourselves to the admissible functions $\psi_s$ in the last four Sections: 
\begin{itemize}
	\item $s$-ellipsoidal function {\bf instead of} $\psi_s$-ellipsoidal function;
	\item $\sell{\upthing{E}}$ {\bf instead of} $\ell_{\psi_s, \upthing{E}};$
	\item $\ovolbs$ {\bf instead of}  $\psiovolb.$
\end{itemize}
That is, for any $d$-ellipsoid $\upthing{E}\in\ellips$ and the $s$-ellipsoidal function $\sell{\upthing{E}}$
identity~\eqref{eq:psi_outer_volume} takes the form
\begin{equation}\label{eq:s_outer_volume}
	\int_{\R^d} \sell{\upthing{E}}=\frac{\alpha}{\det A}\ovolbs.
\end{equation}

\section{Classes of \tpdfpsi-ellipsoidal functions}\label{sec:classes}

\subsection{Boundedness}
In this subsection we show that for any admissible function $\psi$,
we can restrict ourselves to the bounded in $\ellipsd$ set of parameters in problem~\eqref{eq:psi_problem}.
The main result of this subsection is the following lemma.
\begin{lemma}\label{lem:psi_boundedness}
	For any $f\colon\Red\to[0,\infty)$ be a proper log-concave function and $\delta>0$
	there exist  $\rho,\nu,\mu_1,\mu_2>0$
	with the following property.
	If for an admissible function
	$\psi\colon [0, \infty) \to \R \cup \{+\infty\}$
	with 
	$\psi(0)=0$ and $\upthing{E}=(A\oplus\alpha,a)\in\ellips$
	inequalities
		$$f \leq \psiell\quad \text{and}\quad \int_{\R^d} \psiell\leq\delta$$
	hold, then 	inequalities 
		\begin{equation}\label{eq:alpha_a_universal_bound}
			\enorm{a} \leq \rho
			\quad \text{and} \quad 
			\norm{f} \leq \alpha \leq \nu
		\end{equation}
	and
		\begin{equation} \label{eq:comparison_operator}
			\mu_1 \frac{\psiovolb}{
			\left(\lambda_1[\psi]\right)^{d-1}}\cdot \id \prec A \prec \mu_2 \lambda_1[\psi]\cdot \id
		\end{equation}
	hold.
\end{lemma}
\begin{proof}
	% Denote $ \tilde{\alpha} = \alpha e^{-\psi(0)}.$
	Put $\Theta = \norm{f}/2$ and consider 
	the superlevel set $ [f\geq \Theta].$
	Since $f$ is log-concave and has finite positive integral, 
	$[f\geq  \Theta]$ is a convex body in $\R^d.$ 
	Without loss of generality,
	we assume that $\vartheta \ball{d} \subset [f\geq \Theta]$
	for some $\vartheta > 0.$
	%Therefore, $f \geq ||f||/2 \cdot \chi_{\vartheta \ball{d}} .$ 

	Since $f \leq \psiell,$ one sees that 
	$[f\geq \Theta] \subset [ \psiell \geq \Theta].$
	By the log-concavity of $\psiell$ and since the maximum of $\psiell$ is attained at $a,$ we have 
	$ \operatorname{co}\{\vartheta \ball{d}, a\} \subset [ \psiell \geq \Theta]
	.$ Hence,
	\[
		\delta \geq \int_{\R^d} \psiell \geq 
		\Theta \int_{\R^d}  \chi_{\operatorname{co}\{\vartheta \ball{d}, a\}} =
		\Theta \vol{d} \operatorname{co}\{\vartheta \ball{d}, a\} \geq 
		\Theta \frac{\vartheta^{d-1} \vol{d-1} \ball{d-1}}{d}   \enorm{a}.
	\]
	Thus there exists $\rho > 0$ depending only on $f$ and $\delta$ such that $|a|\leq \rho$.

	Clearly,  $\alpha \geq \norm{f}.$
	We proceed with an upper bound on $\alpha$.
	Using again the inclusion  $ \vartheta \ball{d} \subset [f\geq \Theta] \subset [ \psiell \geq \Theta]$
	and by symmetry, we conclude that
	$\vartheta \ball{d}+a \subset [\psiell \geq \Theta].$ 
	Therefore, for all $x\in \vartheta \ball{d}$
	we have
	\[
		\psiell(x+a)
		\geq \psiell(a)^{ 1 - \frac{\enorm{x}}{\vartheta} }
			\cdot \psiell\left(a + \vartheta  \frac{x}{|x|}\right)^{ \frac{\enorm{x}}{\vartheta} }
		\geq \alpha^{1 - \frac{\enorm{x}}{\vartheta} }
			\cdot \Theta^{ \frac{\enorm{x}}{\vartheta} }
		\geq \alpha
			\cdot \left(\frac{\Theta}{\alpha}\right)^{\frac{\enorm{x}}{\vartheta} }.
	\] 
	Therefore, 
	\[
		\delta
			\geq \int_{\R^d} \psiell
			\geq \int_{\vartheta \ball{d}} \psiell(x+a) \di x  
			\geq \alpha \int_{\vartheta \ball{d}} \left(\frac{\Theta}{\alpha}\right)^{\frac{\enorm{x}}{\vartheta} } \di x .
	\]
	By routine computation, the right-hand side tends to $+\infty$
	as $\alpha \to \infty.$
	The existence of $\nu$ follows. 
	% depending only on $f$ and $\delta$ such that
	%$\alpha \leq \nu.$ 
	This completes the proof of inequality~\eqref{eq:alpha_a_universal_bound}.

	We proceed with inequality~\eqref{eq:comparison_operator}.
	%For any line $l$ in $\R^d$ passing through the point $a$ we have
	%\[
	%	\int_{l}\psiell
	%	\geq \int_{l} \frac{||f||}{2}\cdot\chi_{\vartheta \ball{d} + a}
	%	= \frac{||f||}{2} \vartheta.
	%\]
	Let $l$ be the line passing through $a$ in the direction of an eigenvector of $A$ corresponding to the eigenvalue $\norm{A}.$ 
	Since $\vartheta \ball{d}+a \subset  [\psiell \geq \Theta].$ and by~\eqref{eq:psi_outer_volume}, we have
	\[
		2 \Theta \vartheta \leq \int_{l}\psiell=\frac{\alpha}{\norm{A}}\lambda_1[\psi]\leq \frac{\nu }{\norm{A}} \lambda_1[\psi].
	\]
	Thus, $\norm{A} \leq \mu_2\lambda_1[\psi]$ for some $\mu_2>0$.

	Let $\beta$ be the smallest eigenvalue of $A$. Using the obtained bound on $\norm{A}$ and identity~\eqref{eq:psi_outer_volume}, we get
	\[
		\delta 
		\geq \int_{\R^d}\psiell
		= \frac{\alpha}{\det A} \psiovolb
		\geq \frac{\alpha}{\norm{A}^{d-1} \beta} \psiovolb
		\geq \frac{\alpha}{\left( \mu_2 \lambda_1[\psi]\right)^{d-1}} \cdot \frac{1}{ \beta} \psiovolb.
	\]
	%\[
	%	% \frac{\norm{f}}{\delta} \cdot \psiovolb \leq \det A \leq \beta \norm{A}^{d-1}.
	%\]
	Thus, $\beta \geq \mu_1 \psiovolb/\lambda_1[\psi]^{d-1}$ for some $\mu_1>0$.
	This completes the proof.
\end{proof}

As an immediate consequence of Lemma~\ref{lem:psi_boundedness}, we have
\begin{cor}\label{cor:existence}
	Let $\psi\colon [0, \infty) \to \R \cup \{+\infty\}$ be an admissible function
	and let $f\colon \R^d \to [0, \infty)$ be a proper log-concave function.
	If there exists $\ell \in \funclass{\psi}$ such that $f \leq \ell$,
	then there exists a solution to problem~\eqref{eq:psi_problem}.
\end{cor}
\begin{proof}
	If there is $\ell \in \funclass{\psi}$
	such that $f \leq \ell,$
	then, by  Lemma~\ref{lem:psi_boundedness},
	there is a minimizing sequence $\left\{\upthing{E}_i\right\}_1^{\infty}\subset\ellipsd$
	such that
	$f \leq \psiell[\upthing{E}_i]$ for all $i\in\N$,
	\begin{equation} \label{eq:cor_limit_general}
		\lim\limits_{i \to \infty} \upthing{E}_i= \upthing{E} \in \ellipsd
	\end{equation}
	and
	\begin{equation}\label{eq:cor_limit_integrals}
		\lim\limits_{i \to \infty} \int_{\R^d} \psiell[\upthing{E}_i]
		= \inf\limits_{\ell \in \funclass{\psi} } \left\{ \int_{\R^d} \ell  \st f \leq \ell \right\}.
	\end{equation}
	Next, by continuity of a convex function $\psi$ on the interior of its effective domain,
	identity~\eqref{eq:cor_limit_general} implies that
	$\psiell[\upthing{E}_i] \to 
		\psiell[\upthing{E}]$ 
	on the interior of the support of $\psiell[\upthing{E}].$
	Thus, by assertion~\eqref{ass:lem_conv_dual_convergence_integrals} of Lemma~\ref{lem:convergence_dual}, 
	$\psiell[\upthing{E}]$ 
	is a solution to problem~\eqref{eq:psi_problem}.
\end{proof}

\subsection{Interpolation between \tpdfpsi-ellipsoidal functions}

\begin{lemma}[Interpolated $\psi$-ellipsoidal function]\label{lem:psi_containment}
	Let $\psi\colon [0, \infty) \to \R \cup \{+\infty\}$  be an admissible function.
	Let $f_1$ and $f_2$ be proper log-concave functions on $\R^d$.
	Let $\upthing{E}_1=(A_1\oplus\alpha_1,a_1)$ and $\upthing{E}_2=(A_2\oplus\alpha_2,a_2)$
	be $d$-ellipsoids in $\ellipsd$ such that 
		$f_1 \leq \psiell[\upthing{E}_1]$
	and 
		$f_2 \leq \psiell[\upthing{E}_2]$.
	Let $\beta_1,\beta_2>0$ be such that $\beta_1 + \beta_2 = 1$.
	Put 
		\begin{equation}\label{eq:interpolation_def}
			A=\beta_1 A_1 + \beta_2 A_2, \quad
			\alpha=\alpha_1^{\beta_1}\alpha_2^{\beta_2}, \quad 
			a= A^{-1} (\beta_1 A_1 a_1 + \beta_2 A_2 a_2)
			\quad\text{and}\quad
			\upthing{E} = (A\oplus\alpha,a).
		\end{equation}
	Then $\psiell$ satisfies the  following inequalities:
		\begin{equation} \label{eq:psi_int_bound}
			f_1^{\beta_1}f_2^{\beta_2} \leq \psiell
		\end{equation}
	and
		\begin{equation}\label{eq:psi_interpolated_outer_volume}
			\int \psiell
				\leq 
				\left(\int_{\R^d} \psiell[{\upthing{E}_1}]\right)^{\beta_1}
				\left(\int_{\R^d} \psiell[{\upthing{E}_1}]\right)^{\beta_2},
		\end{equation}	 
	with equality in~\eqref{eq:psi_interpolated_outer_volume} if and only if $A_1=A_2.$		
\end{lemma}

\begin{proof}
	Since $f_1 \leq \psiell[\upthing{E}_1]$ and $f_2 \leq \psiell[\upthing{E}_2],$ we have  
		\begin{equation}\label{eq:psi_interp_lemma_f_in}
			f_1^{\beta_1}f_2^{\beta_2}
				\leq 
				\left(\psiell[\upthing{E}_1]\right)^{\beta_1}
				\left(\psiell[\upthing{E}_2]\right)^{\beta_2}.
		\end{equation}	
	Since $\psiell[\ball{d+1}]$ is log-concave and $\beta_1 A_1(x - a_1)+\beta_2 A_2(x - a_2)=A(x-a),$ inequality
	\[
		\psiellf[\ball{d+1}]{A(x-a)}
			\geq 
			\left(\psiellf[\ball{d+1}]{A_1(x - a_1)}\right)^{\beta_1}
			\left(\psiellf[\ball{d+1}]{A_2(x -a_2)}\right)^{\beta_2}
	\]
	holds for all $x \in \Red.$
	Using the identities~\eqref{eq:psi_st_form} and 
	$\alpha = \alpha_1^{\beta_1}\alpha_2^{\beta_2}$, we get 
	$\left( \psiell[{\upthing{E}_1}]\right)^{\beta_1}
		\left( \psiell[{\upthing{E}_2}]\right)^{\beta_2}
			\leq \psiell.$
	This and inequality~\eqref{eq:psi_interp_lemma_f_in} imply inequality~\eqref{eq:psi_int_bound}.

	Using~\eqref{eq:psi_outer_volume},
	we see that~\eqref{eq:psi_interpolated_outer_volume}
	is equivalent to inequality
	\[
		\frac{\psiovolb}{
			\det(\beta_1 A_1+\beta_2 A_2)}\leq 
			\frac{\psiovolb}{
			(\det A_1)^{\beta_1}(\det A_2)^{\beta_2}},
	\]
	which is equivalent to
	\[
		(\det A_1)^{\beta_1}(\det A_2)^{\beta_2} \leq 
		\det(\beta_1 A_1+\beta_2 A_2).
	\]
	That is, inequality~\eqref{eq:psi_interpolated_outer_volume} and the equality condition in it follow from Minkowski’s determinant inequality~\eqref{eq:minkowski_det_multipl_ineq}
	and the equality condition therein.
\end{proof}

Lemma~\ref{lem:psi_containment} and Corollary~\ref{cor:existence} imply
\begin{thm}\label{thm:uniqueness_upto_translation}
	Let  $\psi\colon [0, \infty) \to \R \cup \{+\infty\}$ be an admissible function 
	and let $f$ be a proper log-concave function.
	If there exists $\ell \in \funclass{\psi}$ such that $f \leq \ell$,
	then there exists a solution to problem~\eqref{eq:psi_problem}.
	Moreover, all solutions to this problem are translates of each other.
\end{thm}
\begin{proof}
	By Corollary~\ref{cor:existence}, if the set 
	$\left\{\ell \in \funclass{\psi} \st f \leq \ell\right\}$
	is nonempty,
	then the minimum in~\eqref{eq:psi_problem} 
	is attained at some $\psi$-ellipsoidal function.
	Let $\upthing{E}_1=(A_1\oplus\alpha_1,a_1)$ and $\upthing{E}_2=(A_2\oplus\alpha_2,a_2)$
	be $d$-ellipsoids in $\ellipsd$
	such that $ \psiell[\upthing{E}_1]$ and $\psiell[\upthing{E}_2]$ 
	are the solutions to problem~\eqref{eq:psi_problem}.

	Set $\beta_1 =\beta_2 = 1/2$
	and let $\upthing{E}$ be given by~\eqref{eq:interpolation_def}.
	By the choice of the ellipsoids and~Lemma~\ref{lem:psi_containment},
	$\psiell$ is a $\psi$-ellipsoidal function
	such that $f \leq \psiell$  and 
	\[
		\int_{\R^d} \psiell[\upthing{E}_1]
			\leq \int_{\R^d} \psiell
			\leq \sqrt{\int_{\R^d} \psiell[\upthing{E}_1] \cdot \int_{\R^d} \psiell[\upthing{E}_2]}
			= \int_{\R^d} \psiell[{\upthing{E}_1}].
	\]
	Hence, by the equality condition in~\eqref{eq:psi_interpolated_outer_volume},
	we have $A_1=A_2$.
	Since the integrals of $\psiell[{\upthing{E}_1}]$ and $\psiell[{\upthing{E}_2}]$ are equal
	and by~\eqref{eq:psi_outer_volume}, 
	we obtain $\alpha_1=\alpha_2$.
	This completes the proof.
\end{proof}

\section{When is the solution  unique?}\label{sec:uniqueness}

Let $\ell_1$ and $\ell_2$ be two $\psi$-ellipsoidal function that are translates of each other. 
Summarizing the already proved results,
we need to be able to ``squeeze'' a $\psi$-ellipsoidal function $\ell$ between $\ell_1$ and $\ell_2$
in such a way that
$\int_{\R^d}  \ell <  \int_{\R^d}  \ell_1 =  \int_{\R^d}  \ell_2$
and $\min\{\ell_1, \ell_2\} \leq \ell$
to show the uniqueness of a solution to problem~\eqref{eq:psi_problem}. 
We are able to find such an interpolation for two classes  of admissible functions.
They are the class of strictly increasing admissible functions
and the class of admissible functions with bounded effective domain.
We will prove the following two lemmas about interpolation between two translated $\psi$-ellipsoidal functions:

\begin{lemma}\label{lem:psi_sausage_lemma_st_increasing}
	Let $\psi\colon [0, \infty) \to \R \cup \{+\infty\}$ 
	be a strictly increasing admissible function.
	Let $a_1 \neq a_2$ and let $\upthing{E}_1=(A\oplus\alpha, a_1)$
	and $\upthing{E}_2=(A\oplus\alpha,a_2)$ be elements of $\ellipsd.$
	Then there exists $\upthing{E} \in \ellipsd$
	such that
	\[
	  \int_{\R^d} \psiell  
		< \int_{\R^d} \psiell[\upthing{E}_1]
		= \int_{\R^d} \psiell[\upthing{E}_2]
		\quad  \text{and} \quad
		\min\left\{ \psiell[\upthing{E}_1], \psiell[\upthing{E}_2]\right\} \leq \psiell.
	\]
\end{lemma}

\begin{lemma}\label{lem:psi_sausage_lemma_bounded_domain}
	Let  $\psi\colon [0, \infty) \to \R \cup \{+\infty\}$
	be an admissible function with bounded effective domain.
	Let $a_1 \neq a_2$ and
	let $\upthing{E}_1=(A\oplus\alpha, a_1)$
	and $\upthing{E}_2=(A\oplus\alpha,a_2)$
	be elements of $\ellipsd.$
	Then there exists 
	$\upthing{E} \in \ellipsd$ such that 
	\[
		\int_{\R^d} \psiell
		< \int_{\R^d} \psiell[\upthing{E}_1]
	 	= \int_{\R^d} \psiell[\upthing{E}_2]
	 	\quad  \text{and} \quad
		\min\left\{ \psiell[\upthing{E}_1], \psiell[\upthing{E}_2]\right\} \leq \psiell.
	\] 
\end{lemma}
We use two different ideas of how to construct a $\psi$-ellipsoidal function $\psiell$ in these lemmas.

Before we proceed with the proofs of Lemmas~\ref{lem:psi_sausage_lemma_st_increasing} and~\ref{lem:psi_sausage_lemma_bounded_domain},
let us show how they imply the uniqueness result.
\begin{thm}\label{thm:lowner_uniqueness_general}
	Let  $\psi\colon [0, \infty) \to \R \cup \{+\infty\}$ be an admissible function.
	Additionally,
	let $\psi$ be a strictly increasing function
	or let the effective domain of $\psi$ be bounded. 
	Let $f\colon \R^d \to [0, \infty)$ be a proper log-concave function.
	If there exists $\ell \in \funclass{\psi}$ such that $f \leq \ell$,
	then there exists a unique solution to problem~\eqref{eq:psi_problem}.
\end{thm}
\begin{proof}
	Assume there exists $\ell \in \funclass{\psi}$ such that $f \leq \ell.$
	By Corollary~\ref{cor:existence}, the solution to problem~\eqref{eq:psi_problem} exists.
	Assume that there are at least two different solutions $\ell_1$ and $\ell_2.$
	By construction, we have that
	$f \leq \min\{\ell_1, \ell_2\}.$
	This means that a $\psi$-ellipsoidal function
	$\ell$ such that $\min\{\ell_1, \ell_2\} \leq \ell$
	also satisfies inequality $f \leq \ell.$
	By Theorem~\ref{thm:uniqueness_upto_translation}, 
	the $\psi$-ellipsoidal functions $\ell_1$ and $\ell_2$ are translates of each other.
	Hence, by Lemma~\ref{lem:psi_sausage_lemma_st_increasing} (in case $\psi$ is strictly increasing)
	and by Lemma~\ref{lem:psi_sausage_lemma_bounded_domain} (in case $\psi$ has bounded effective domain),
	$\ell_1$ and $\ell_2$ are not solutions to problem~\eqref{eq:psi_problem}.
	This contradicts the choice of $\ell_1$ and $\ell_2.$
\end{proof}

\subsection{Proofs of Lemma~\ref{lem:psi_sausage_lemma_st_increasing} and Lemma~\ref{lem:psi_sausage_lemma_bounded_domain}}

For the sake of simplicity,  we assume that $\psi(0)=0.$
Without loss of generality, we assume that $a_1$ and $a_2$
are the opposite vectors, that is, $a_2= -a_1 \neq 0.$

Set $a = - A a_1.$
Consider the half-space 
	$H^{+} = \{x \in \R^d \st \iprod{A x}{ a} \geq 0\}.$ 
Clearly, we have	
\begin{equation}\label{eq:sausage_two_func}
	\psiellf[\upthing{E}_1]{x}
		\leq \psiellf[\upthing{E}_2]{x}
		\quad\text{for all}
		\ x \in H^{+}.
\end{equation}

\begin{proof}[Proof of Lemma~\ref{lem:psi_sausage_lemma_st_increasing}]
	The idea of our construction is as follows.
	We consider $\psi$-ellipsoidal functions $\psiell,$
	where $\upthing{E}$ is of the form  
	$\left( \rho_1 A \oplus \rho_2 \alpha, 0\right).$ 
	Next, we find suitable $\rho_1, \rho_2 \in (0,1).$ 
	That is, we decrease the height {\bf and} the determinant of the operator. 

	First, we construct $\psi$-ellipsoidal functions which satisfy inequality
	$\min\left\{\psiell[\upthing{E}_1],  \psiell[\upthing{E}_2]\right\}\leq \psiell.$
	\begin{claim}
		Set
		\begin{equation}\label{eq:sausage_eps_def}
			\epsilon_1 =  \frac{1}{2d}{\inf
				\left\{\psi(\enorm{Ax+a})-
			\psi\!\left(\enorm{Ax}\right) \right\}},
		\end{equation}
		where the infimum is taken over $x \in H^+$
		satisfying $\psi\!\left(\enorm{Ax}\right) \leq 2d .$
		Then $\epsilon_1 > 0$ and for any 
		$\epsilon \in [0, \min\{\epsilon_1 , 1\}),$ 
		the $\psi$-ellipsoidal function $\psiell$ with 
		\begin{equation}
		\label{eq:sausage_E_def}
			\upthing{E}=\left((1-\epsilon)A\oplus\alpha 
			e^{-2d \cdot \epsilon}, 0\right),
		\end{equation}
		satisfies inequality $\min\left\{\psiell[\upthing{E}_1], \psiell[\upthing{E}_2]\right\} \leq \psiell.$
	\end{claim}
	\begin{proof}
		Let us show that $\epsilon_1 > 0.$
		Since $\psi$ is a convex function, for any $\tau, t>0$ we have 
		\[
			\psi(\tau) = \psi(\tau) - \psi(0) \leq \psi(\tau +t) - \psi(t).
		\]
		Hence, 
		$ \psi\!\left(\enorm{Ax +a} - \enorm{Ax}\right)
			\leq \psi(\enorm{Ax+a})-\psi\!\left(\enorm{Ax}\right).$
		Since $\psi$ is a convex function and since
		$\lim\limits_{t \to \infty} \psi(t) = \infty,$
		the set 
		$\left\{x \in \R^d \st \psi\!\left(\enorm{Ax}\right)
			\leq 2d\right\}$
		is bounded. Clearly, $\enorm{Ax +a} - \enorm{Ax}$ is bounded from below by some positive constant
		on the bounded set $H^{+} \cap \{x \st \psi\!\left(\left|Ax\right|\right) \leq  2d\}.$
		Since $\psi$ is strictly increasing,
		we conclude that $\epsilon_1 > 0.$ 

		Let us check inequality  
		$\min\left\{ \psiell[\upthing{E}_1],\psiell[\upthing{E}_2]\right\}
			\leq \psiell.$
		Using~\eqref{eq:psi} and~\eqref{eq:sausage_two_func}, we see that it suffices to prove that
			\begin{equation*}
				\psi\left(\enorm{(1-\epsilon)Ax}\right)+ 2d \cdot \epsilon
					\leq \psi(\enorm{Ax+a})
					\quad \text{for all}
					\ x \in H^{+}.
			\end{equation*}
		Since $\psi$ is a convex function, for any $\epsilon \in [0,1)$, we have
			\begin{equation*}
				\psi\left(|(1-\epsilon)Ax|\right)
				\leq (1-\epsilon)\psi(\enorm{Ax})+\epsilon \psi(0)= (1-\epsilon)\psi(\enorm{Ax}).
			\end{equation*}
		Therefore, it suffices to show that
			\begin{equation}\label{eq:psi_sausage_lem}
				(1-\epsilon)\psi(\enorm{Ax})+ 2d \cdot \epsilon \leq \psi(\enorm{Ax+a})
				\quad\text{for all}
				\ x \in H^{+}.
			\end{equation}
		Consider three cases.
		First, if $\psi(\enorm{Ax}) = + \infty$ for $x\in H^+$, then $\psi(\enorm{Ax+a}) = + \infty$.
		Therefore, inequality~\eqref{eq:psi_sausage_lem} trivially holds in this case.
		Second, since $\enorm{Ax+a} > \enorm{Ax}$ for any $x\in H^+$, 
		inequality~\eqref{eq:psi_sausage_lem} holds for any non-negative~$\epsilon$
		and any $x$ in $\{y \in H^+ \st \psi\!\left(\left|Ay\right|\right) \geq 2d\}.$	
		Third, consider $x\in \{y \in H^+ \st \psi\!\left(\left|Ay\right|\right) \leq 2d\}$.
	 	Then inequality~\eqref{eq:psi_sausage_lem} takes the form
			\begin{equation}\label{eq:psi_sausage_lemma_2}
				\epsilon(2d - \psi(\enorm{Ax}))\leq 
				\psi\!\left(\enorm{Ax+a}\right)-\psi\!\left(\enorm{Ax}\right).
			\end{equation}
		Set 
		\[
			\epsilon_0 = \inf\limits_{x \in H^{+} \st\, \psi\,\!\left(\enorm{Ax}\right) \leq 2d}\left\{\frac{\psi(\enorm{Ax+a})-
			\psi\!\left(\enorm{Ax}\right)}{2d - \psi\!\left(\enorm{Ax}\right)}\right\}.
		\]
		We see that inequality~\eqref{eq:psi_sausage_lemma_2} holds
		for any  $\epsilon \in [0, \epsilon_0].$
		Hence, inequality~\eqref{eq:psi_sausage_lem} holds
		for any $\epsilon\in [0,\min\{1,\epsilon_0\}]$ and any $x\in H^+$.
		Clearly, $\epsilon_0  \geq \epsilon_1.$
		By symmetry, $\epsilon_1$ given by~\eqref{eq:sausage_eps_def} satisfies the required property.
	\end{proof}

	Using~\eqref{eq:psi_outer_volume} with $\upthing{E}$ given by~\eqref{eq:sausage_E_def}, we obtain 
	\[
		\int_{\R^d} \psiell = 
		\frac{ e^{- 2d \cdot \epsilon}}{(1-\epsilon)^d} \cdot \int_{\R^d} \psiell[\upthing{E}_1].
	\]
	However, since 
	$\frac{ e^{- 2d \cdot \epsilon}}{(1-\epsilon)^d} < 1$ for any $\epsilon \in (0,1/2],$
	we conclude that  $\int_{\R^d} \psiell < \int_{\R^d} \psiell[\upthing{E}_1]$
	for a sufficiently small positive $\epsilon.$ 
	This completes the proof of Lemma~\ref{lem:psi_sausage_lemma_st_increasing}.
\end{proof}
%\begin{proof}[Proof of \Href{Theorem}{thm:psi_lowner}]
%By Theorem~\ref{thm:uniqueness_upto_translation}, we need to show 
%\TODO{hm Mojno v nachalo sekcii napisat' obshee rassujdenie. esli mojem vpichnut' to edinstvennoe reshenie.}
%	We show that such a function is unique. Let $\upthing{E}_1=(A\oplus\alpha,a_1)$ and $\upthing{E}_2=(A\oplus\alpha,a_2)$ be ellipsoids such that $f\leq \psiell[\upthing{E}_1]$ and 
%	$f \leq \psiell[\upthing{E}_2]$ and the $\psi$-ellipsoidal functions $\psiell[\upthing{E}_1]$ and 
%	$\psiell[\upthing{E}_2]$ are of minimal integral.
% Assume that $a_1\neq a_2$. By translating the origin, we may assume that $a_1=-a_2$. Applying Lemma~\ref{lem:psi_sausage_lemma}, we get that for some $\epsilon_1 > 0$ and 
%	$\epsilon = \min\{1/2, \epsilon_1\},$ the following  inequality 
%	\[
%	f \leq \psiell[\upthing{E}_\epsilon] \quad
%	\text{with} \quad 
%\upthing{E}_\epsilon = 
%\left(\left(1-\epsilon\right)A_1 \oplus \alpha_1 e^{-2d \cdot \epsilon }, 0\right)
%\]
%holds.   However,  since 
%$\frac{ e^{- 2d \cdot \epsilon}}{(1-\epsilon)^d} < 1$ for any $\epsilon \in (0,1/2]$ and  
%using~\eqref{eq:psi_outer_volume}, we obtain that
%\[
%\int \psiell[\upthing{E}_\epsilon] = 
%\frac{ e^{- 2d \cdot \epsilon}}{(1-\epsilon)^d} \cdot \int \psiell 
% < \int \psiell.
%\]
%This contradicts the choice of $\upthing{E}.$ Theorem~\ref{thm:psi_lowner} is proven.
%\end{proof}
\begin{proof}[Proof of Lemma~\ref{lem:psi_sausage_lemma_bounded_domain}]
	Here we use the same idea as in the setting of convex sets, that is, 
	if the set $\left(\ball{d} + a_1\right) \cap  \left(\ball{d} + a_2\right)$
	with $a_1 \neq a_2$ is nonempty,
	then it is a subset of  $\rho \ball{d} + \frac{a_1 + a_2}{2}$ for some $\rho \in (0,1).$
	We apply this observation for the superlevel sets of $\psi$-ellipsoidal functions and
	consider $\psi$-ellipsoidal functions $\psiell,$
	where $\upthing{E}$ is of the form $\left(A_1 \oplus \alpha, 0\right)$ with $A \prec A_1.$
	The boundedness of the effective domain allows us to show that there exists a suitable $A_1.$

	Denote $\delta = |a|.$ 
	Let $[0,\tau]$ be the closure of $\dom\psi$.
	We have that $\supp \psiell[\upthing{E}_1] =\{x\in\R^d \st\enorm{Ax +a} \leq \tau\}.$
	Hence, the interiors of supports of $\psiell[\upthing{E}_1]$ and 
	$\psiell[\upthing{E}_2]$ do not intersect if $\delta \geq \tau.$
	Thus, any $\upthing{E}$ of the form  
	$\left( \rho A \oplus \alpha, 0\right)$ with any $\rho > 1$ suits us in this case.

	Assume that $\delta < \tau.$
	Choose an orthonormal basis of $\R^d$
	such that the first vector of the basis is in the direction of the vector $a$.
	We claim that the function $\psiell$ with
	\[
		\upthing{E} =
			(MA\oplus\alpha,0),
			\quad\text{where}\quad
			M=\diag\left\{\frac{\tau}{\tau-\delta},1,\dots,1\right\},
	\]
	satisfies the condition of the lemma.

	The bound on the integral of $\psiell$ trivially holds. 

	Let us check inequality $\min\left\{ \psiell[\upthing{E}_1], \psiell[\upthing{E}_2]\right\}\leq \psiell.$
	By~\eqref{eq:psi} and~\eqref{eq:sausage_two_func}, it suffices to prove that
		\begin{equation}\label{eq:bounded_domain_equiv}
			\psi\left(|MAx|\right)
			\leq \psi(\enorm{Ax+a})
			\quad\text{for all}\ 
			x \in H^{+}\cap\supp{\psiell[\upthing{E}_1]}.
		\end{equation}
	Since an admissible function is increasing on its domain,
	inequality~\eqref{eq:bounded_domain_equiv} follows from the following:
		\begin{equation}\label{eq:bounded_domain_1}
			\enorm{MAx} 	
				\leq \enorm{Ax+a}
				\quad \text{for all}
				\ x \in H^{+}\cap\{v \in \R^d \st \enorm{Av+a} \leq \tau\}.
		\end{equation}
	Indeed, denote $y=Ax$ and let $y=\lambda a+z,$ 
	where $z$ is orthogonal to $a.$ 
	We have 
	\[
		\enorm{My}^2=\enorm{M(\lambda a + z)}^2 
			= \enorm{\frac{\tau}{\tau - \delta}\lambda a+z}^2
			= \left(\frac{\tau}{\tau - \delta}\right)^2 \lambda^2 \delta^2 + \enorm{z}^2
	\]			
	and $\enorm{y+a}^2=(1+\lambda)^2\delta^2 + \enorm{z}^2$. 
	Since  $\lambda\geq 0$ for  $x\in H^+,$  inequality~\eqref{eq:bounded_domain_1} is equivalent to
		$$\frac{\tau}{\tau - \delta} \lambda
			\leq 1 + \lambda,$$
	which trivially holds since $(1+\lambda)\delta \leq \enorm{y + a} = \enorm{Ax + a} \leq \tau.$
	Thus, inequality~\eqref{eq:bounded_domain_equiv} holds.
	We conclude that $\min\left\{ \psiell[\upthing{E}_1], \psiell[\upthing{E}_2]\right\}\leq \psiell,$
	completing the proof of Lemma~\ref{lem:psi_sausage_lemma_bounded_domain}. 
\end{proof}
	% Since $\iprod{Ax}{a}\geq 0$, we have $\lambda>0$.
	% The inequality $|Ax+a|\leq\tau$ implies that $(\lambda+1)|a|\leq\tau$
%\end{proof}

\subsection{Chimeras}\label{subsec:chimeras}

Let $\psi\colon [0, \infty) \to \R $
be an admissible function such that
it is constant on some interval $[0,\tau], \tau > 0,$
and it takes only finite values.
We claim that,
given such an admissible function and a log-concave function $f$,
a solution to problem~\eqref{eq:psi_problem} is not necessarily unique,
which we show in the following example.
For simplicity, we assume that $\psi(0)=0.$
\begin{exl}
	Let $f\colon \R^d \to [0,\infty)$ be given by 
	$f = \min\left\{\psiell[(\id \oplus 1, -c)],\psiell[(\id \oplus 1, c)] \right\},$
	where $|c| = \tau.$ Then the functions of the form
	$\psiell[(\id \oplus 1,\rho c)]$ with $\enorm{\rho} \leq 1$
	are solutions to problem~\eqref{eq:psi_problem}.
\end{exl}
\begin{proof}
	By monotonicity of the Euclidean norm, we see that  any function of the form
	$\psiell[(\id \oplus 1,\rho c)]$ with $|\rho| \leq 1$  satisfies inequality
	$f \leq \psiell[(\id \oplus 1,\rho c)].$
	Thus, we need to show that these functions are of minimal integral.

	Let $(A \oplus \alpha, a) \in \ellipsd$ be such that 
	$f \leq \psiell[(A \oplus \alpha, a)].$
	Since $f(0) = 1,$ we conclude that $\alpha \geq 1.$
	By this and by~\eqref{eq:psi_outer_volume}, it suffices to show that $\norm{A} \leq 1.$
	Assume the contrary: There is an eigenvalue~$\lambda$ of~$A$ such that $\lambda > 1.$
	Denote by~$u$ a unit eigenvector corresponding to $\lambda$ such that $\iprod{u}{c} \geq 0.$
	Then, for all  $\xi > 0,$ the following inequality holds:
	\[
		\alpha e^{-\psi(|\lambda \xi u - A a|)} \geq e^{-\psi(|\xi u + c|)},
	\] 
	or, equivalently,
	\begin{equation}\label{eq:ploskii_modulus_exl}
		\psi(\enorm{\xi u + c})  -
		\psi(\enorm{\lambda \xi u  - A a}) \geq - \ln  \alpha.
	\end{equation}
	By convexity of $\psi,$  the left-hand side in~\eqref{eq:ploskii_modulus_exl} is at most
	$-\psi(\enorm{\lambda \xi u - A a} - \enorm{\xi u + c}).$ 
	Since $\lambda >1,$ the argument of $\psi$ tends to infinity as $\xi$ tends to infinity. 
	Thus, the left-hand side in~\eqref{eq:ploskii_modulus_exl} is strictly less than
	$-\ln \alpha$
	for a sufficiently large $\xi.$ We obtain a contradiction.
\end{proof}

\section{When does the solution exist?}\label{sec:existence}

In this Section,
we address the question of existence of a solution to problem~\eqref{eq:psi_problem}. 
By Corollary~\ref{cor:existence},
it suffices to find one $\psi$-ellipsoidal function $\ell$ such that $f \leq \ell.$
This is a simple technical question.
\begin{exl}
	Let  $\psi\colon [0, \infty) \to \R \cup \{+\infty\}$  be  an admissible function  which is not of  linear growth.
	Then for any $d \in \N,$ the set
	\[
		\left\{
			\ell \in \funclass{\psi} \st e^{-\enorm{x}} \leq \ell  
		 \right\} 
	\]
	is empty.
\end{exl}

We show below that if $\psi\colon [0, \infty) \to \R \cup \{+\infty\}$
is an admissible function of linear growth,
then for any proper log-concave function $f\colon \Red \to [0,\infty),$ 
there is a $\psi$-ellipsoidal function $\ell$ such that $f \leq \ell.$  
Next, we bound the norm of the maximizers of~\eqref{eq:psi_problem} in Lemma~\ref{lem:lowner_height_bounded},
and complete the proof of Theorem~\ref{thm:psi_lowner}.

\subsection{Existence of a ``covering''}

\begin{lemma}\label{lem:height_covering}
	Let $\psi\colon [0, \infty) \to \R \cup \{+\infty\}$ 
	be an admissible function
	and $f\colon \R^d \to [0, \infty)$ be a proper log-concave function.
	If there exists $\upthing{E} \in \ellipsd$ such that $f \leq \psiell,$
	then for an arbitrary $\alpha > \norm{f},$ there exists $(A \oplus \alpha, a) \in \ellipsd$
	such that $f \leq \psiell[( A \oplus \alpha, a)].$
\end{lemma}
\begin{proof}
	Without loss of generality, we assume that $\upthing{E} = \ball{d+1}$ and $\psi(0)=0.$
	Case $ \alpha > 1$ is trivial.
	Consider $\norm{f} < \alpha < 1.$ 
	Put $S= \{x \in \R^d \st \psiell[(1/2 \cdot \id \oplus \alpha, 0)] \leq \psiell[\ball{d+1}]\}.$
	Since $\psi$ is convex and admissible, it follows that $S$ is bounded. 
	Clearly, for a sufficiently small $\rho \in (0,1/2)$ we have 
	$\norm{f} \leq \psiell[(\rho \cdot \id \oplus \alpha, 0)](x)$
	for all $x \in S.$
	By monotonicity, 
	$\psiell[(1/2 \cdot \id \oplus \alpha, 0)]
		\leq \psiell[(\rho \cdot \id \oplus \alpha, 0)].$
	Thus, by construction,
	$f \leq \psiell[(\rho \cdot \id \oplus \alpha, 0)].$
\end{proof}

\begin{lemma}\label{lem:psi_lin_growth_bound}
	Let $\psi\colon [0, \infty) \to \R \cup \{+\infty\}$
	be an admissible function of linear growth with $\psi(0)=0$
	and $f\colon \R^d \to [0,\infty)$ be a proper log-concave function.
	Then for any $(A \oplus \alpha, a) \in \ellipsd$ with $\alpha > \norm{f}$,
	there exists $\gamma > 0$ such that $f \leq \psiell[(\gamma A \oplus \alpha, a)].$
\end{lemma}
\begin{proof}
	By monotonicity of $\psi$, it suffices to prove the lemma for $A = \id.$
	Without loss of generality, we assume that $a = 0.$ 
	It is known that (see~\cite[Lemma~2.2.1]{brazitikos2014geometry})
	for any proper log-concave function $f$ on $\Red$, there are $\Theta, \vartheta > 0$ such that 
	\begin{equation*}\label{eq:pointbound}
		f(x) \leq \Theta e^{-\vartheta \enorm{x}}, \text{ for all } x\in\Red.
	\end{equation*} 
	On the other hand, since $\psi$ is convex and of linear growth, 
	there are $\gamma_1 > 0$ and $C>0$ such that 
	\[
		\Theta e^{-\vartheta \enorm{x}} \leq \alpha e^{-\psi(\gamma_1\enorm{x})} 
		\quad \text{for all}\ x \ \text{such that}\ \enorm{x} > C. 
	\] 
	By monotonicity of $\psi,$ this inequality holds for all
	$\gamma \in (0, \gamma_1).$

	However, by continuity,  there exists $\gamma_2 > 0$ such that
	inequality 
	\[
		\norm{f} \leq \alpha e^{-\psi(\gamma_2\enorm{x})} \; \
		\text{holds for all} \ x \ \text{such that}\  \enorm{x} \leq C. 
	\] 
	That is, $\gamma = \min\{\gamma_1, \gamma_2\}$ satisfies the required property.
\end{proof}

As an immediate consequence of Corollary~\ref{cor:existence} and Lemma~\ref{lem:psi_lin_growth_bound}, we get the following.
\begin{cor}\label{cor:psi_lin_lowner}
	Let $\psi\colon [0, \infty) \to \R \cup \{+\infty\}$
	be an admissible function  of linear growth. 
	Let $f\colon \R^d \to [0,\infty)$ be a proper log-concave function.
	Then there exists a solution to problem~\eqref{eq:psi_problem}.
\end{cor}

\subsection{Bound on the height}
The following result
is an extension of the analogous result from~\cite{ivanov2020functional}
about the John $s$-function with a similar proof.
The idea of the  proof  can be traced back to~\cite{alonso2018john}.
\begin{lemma}\label{lem:lowner_height_bounded}
	% Let $L$ be a solution to~\eqref{eq:psi_problem}
	Let $\psi\colon [0, \infty) \to \R \cup \{+\infty\}$ be an admissible function
	and
	$f\colon \R^d \to [0, \infty)$ be a proper log-concave function.
	Then for a solution $L$ to problem~\eqref{eq:psi_problem} we have
		\begin{equation}\label{eq:lowner_func_norm_bound}
			\norm{L} \leq e^d \norm{f}.
		\end{equation}
\end{lemma}
\begin{proof}
	Without loss of generality, we assume that $\psi(0)=0.$
	There is nothing to prove if $\norm{f} = \norm{L}.$
	Assume $\norm{f} < \norm{L}.$
	We define a function $\Psi\colon(\norm{f},+\infty) \to [0,+\infty)$ as follows.
	Let $\alpha > \norm{f}.$
	By Lemma~\ref{lem:height_covering},
	there exists $\upthing{E}_{\alpha} = (A_{\alpha} \oplus \alpha,a) \in \ellipsd$
	such that $f \leq \psiell[\upthing{E}_{\alpha}].$	
	By Lemma~\ref{lem:psi_boundedness} and by compactness,
	we may assume that $\psiell[\upthing{E}_{\alpha}]$ is of minimal integral
	among  $\psi$-ellipsoidal functions~$\ell$
	of height $\alpha$ and such that $f \leq \ell$.
	Put $\Psi(\alpha)=\det A_{\alpha}$.
	The function $\Psi$ is well-defined.
	Indeed, by~\eqref{eq:psi_outer_volume},
	we have that the determinants of operators
	that correspond to two different $\psi$-ellipsoidal functions
	with the same heights and the same integrals,
	are equal.

	\begin{claim}
		For any $\alpha_1, \alpha_2 \in ( \norm{f}, \infty)$ and $\lambda \in [0,1],$ 
		we have
		\begin{equation}\label{eq:lowner_ellips-d_quasi_log-concavity}
			\Psi \!\left(\alpha_1^\lambda \alpha_2^{1-\lambda} \right)^{1/d} \ge \lambda 
			\Psi(\alpha_1)^{1/d} + (1 - \lambda) \Psi(\alpha_2)^{1/d}.
		\end{equation}
	\end{claim}
	\begin{proof}
		Let $\left(A_1 \oplus \alpha_1, a_1 \right),$
		$\left(A_2 \oplus \alpha_2, a_2 \right) \in \ellipsd$  be such that 
		\[
			f \leq \psiell[\left(A_1 \oplus \alpha_1, a_1 \right)]
			\quad \text{and} \quad
			f \leq \psiell[\left(A_2 \oplus \alpha_2, a_2 \right)],
		\]
		and 
		\[
			\Psi(\alpha_1) = \det A_1
			\quad \text{and} \quad
			\Psi(\alpha_2) = \det A_2.
		\]
		By the definition of the function $\Psi$, Lemma~\ref{lem:psi_containment}
		and identity~\eqref{eq:psi_outer_volume},
		we have that
		\[
			\Psi\! \left(\alpha_1^\lambda \alpha_2^{1-\lambda}\right)
			\geq \det \left(\lambda A_1 + (1-\lambda)A_2 \right).
		\]
		Now,~\eqref{eq:lowner_ellips-d_quasi_log-concavity} follows immediately from
		Minkowski's determinant inequality~\eqref{eq:minkowski_det_ineq}.
	\end{proof}

	Set $\Phi(t) = \Psi\! \left(e^{t} \right)^{1/d}$ for all $t \in 
	\left({\log \norm{f}}, +\infty \right).$
	Inequality~\eqref{eq:lowner_ellips-d_quasi_log-concavity} implies
	that $\Phi$ is a concave function on its domain.	
		
	Let $\alpha_0$ be the height of $L$. 
	Then, by~\eqref{eq:psi_outer_volume}, for any $\alpha$ in the domain of $\Psi,$
	we have that 
	\[
		\frac{\alpha_0}{\Psi(\alpha_0)}
		\leq
		\frac{\alpha}{\Psi(\alpha)}.
	\]
	Setting $t_0 =  \log \alpha_0$ and taking root of order $d$, for any $t$ in the domain of $\Phi$ we obtain
	\[
		\Phi(t) \leq \Phi(t_0) e^{\frac{1}{d}(t - t_0)}.
	\]
	The expression on the right-hand side is a 
	convex function of $t$, while $\Phi$ is a concave 
	function. Since these functions take the same value at $t = t_0,$ we conclude 
	that the graph of $\Phi$ lies below the tangent line to the graph of $\Phi(t_0) 
	e^{\frac{1}{d}( t - t_0)}$ at the point $t_0.$ That is, 
	\[
		\Phi(t) \leq \Phi(t_0) \left(1 + \frac{1}{d}(t - t_0)\right).
	\]
	Passing to the limit as $t \to  {\log \norm{f}}$  and since the values 
	of $\Phi$ are  positive, we get
	\[
		0 \leq 1+\frac{\log \norm{f}}{d}-\frac{t_0}{d},
	\]
	or, equivalently, $t_0 \leq  {d} + \log \norm{f}.$
	Therefore, $\alpha_0 = \norm{L} \leq  e^{d} \norm{f}.$
	This completes the proof of Lemma~\ref{lem:lowner_height_bounded}.	
\end{proof}

\subsection{Proof of Theorem~\ref{thm:psi_lowner}}

Since $\funclass{\psi} = \funclass{\psi + c}$ for any constant $c,$
Theorem~\ref{thm:psi_lowner} is an immediate consequence of
Theorem~\ref{thm:lowner_uniqueness_general} and Lemma~\ref{lem:lowner_height_bounded}.

\section{Classes of \tpdfs-ellipsoidal functions}\label{sec:s_lowner}

In this Section, we discuss $s$-ellipsoidal functions and their properties.
We recall the definition of John $s$-function
and explain basic properties of the duality between two optimization problems.
Next, we prove Theorem~\ref{thm:s_lowner} and~\ref{thm:zero_limit}.

\subsection{Basic properties of \tpdfs-ellipsoidal functions and the John \tpdfs-functions}\label{subsec:john_s-func}

Recall that for any $\upthing{E}= (A \oplus \alpha, a) \in \ellipsd$ and 
any $s \in (0, \infty),$ the $s$-ellipsoidal function $\sell{\upthing{E}}$
is given by
	\begin{equation}\label{eq:ellipsoid_function_correspondence}
		\sell{\upthing{E}}(x)=\alpha e^{- \psi_s \left(
		\enorm{A(x-a)}\right)} =
		\alpha\left[\frac{1+\sqrt{1+\frac{4}{s^2} \enorm{A(x-a)}^2}}{2}
		\cdot \exp\left({1-\sqrt{1+\frac{4}{s^2} \enorm{A(x-a)}^2}}\right)\right]^{s/2}.
	\end{equation}
However, we have not shown that $\psi_s$ is an admissible function.
To avoid boring computations,
we show the basic properties of $s$-ellipsoidal functions
using their log-conjugate functions.
At the same time, we reveal the duality between our definition of L\"owner $s$-function
and the definition of John $s$-function introduced in~\cite{ivanov2020functional}.

For any $d$-ellipsoid $\upthing{E}=(A\oplus\alpha,a)\in \ellipsd$ and any $s>0,$ define
\[
	\shf{\upthing{E}} (x) = 
	\begin{cases}
		\cfrac{1}{\alpha}
		\left[1 -\enorm{A^{-1} (x - a)}^2\right]^{s/2},& 
		\text{ if }x \in A \ball{d}+a\\
		0,&\text{ otherwise}.
	\end{cases}
\]
The geometric sense of this function is as follows. 
The graph of  $\shf[1]{\left(A \oplus \alpha,a \right)}$ restricted to its support
coincides with the upper hemisphere of the $d$-ellipsoid $(A \oplus \alpha^{-1/s})\ball{d+1} + a.$
It follows that $\shf{\upthing{E}} \leq f$
if and only if the $d$-ellipsoid $\left(A \oplus 1/\alpha\right)\ball{d+1} + a$
is contained within the subgraph of $f^{1/s}.$ 
Another consequence of the observation that $\shf[1]{\upthing{E}}$ is a
``height'' function of the $d$-ellipsoid is that $\shf[1]{\upthing{E}}$
is concave on its support. Hence, $\shf{\upthing{E}}$ is  log-concave.
Clearly, $\shf{\upthing{E}}$ is a proper log-concave function.

For a $d$-ellipsoid $\upthing{E}=(A\oplus\alpha,a)\in \ellipsd,$ we put
\[
	\shf[0]{\upthing{E}}  = \frac{1}{\alpha}\chi_{A \ball{d}+a}
	\quad \text{and} \quad
	\shf[\infty]{\upthing{E}}  = \frac{1}{\alpha} 
	e^{- \enorm{{A^{-1} (x - a)}}^2}.
\]

The authors of~\cite{ivanov2020functional} consider the problem
\begin{equation}\label{eq:john_problem}
	\max_{\upthing{E}\in\ellipsd}\int_{\R^d} \shf{\upthing{E}}
	\quad \text{subject to}\quad
	\shf{\upthing{E}}\leq f;
\end{equation}
and show that for any fixed $s \in [0,\infty)$
the solution to this problem exists and is unique~\cite[Theorems~4.5 and~7.3]{ivanov2020functional}
for any proper log-concave function $f\colon \R^d \to [0,\infty).$
The case $s=0$ was studied earlier in~\cite{alonso2018john}.
For a proper log-concave function $f\colon \R^d \to [0,\infty),$
we call the unique solution to problem~\eqref{eq:john_problem}
the \emph{John $s$-function} of $f$ and denote it by  $\selldense{f}.$

In the next lemma we prove that for any $s \in [0, \infty]$, the affine classes  
$\left\{\shf{\upthing{E}} \st \upthing{E} \in \ellipsd\right\}$
and 
$\funclass{\psi_s} = \left\{\sell{\upthing{E}} \st \upthing{E} \in \ellipsd\right\}$
are the affine classes of two functions polar to each other. 
That is, problems~\eqref{eq:john_problem} and~\eqref{eq:s_problem} are  dual.
This gives us a reason for considering classes of $s$-ellipsoidal functions.   
\begin{lemma}\label{lem:duality_h_l}
	Let $\upthing{E}=(A\oplus\alpha,0)\in\ellipsd $ and $s\in [0,\infty].$ 
	Then
	\begin{equation}\label{eq:duality_l_j}
		\left(\shf{\upthing{E}}\right)^{\circ}= \sell{\upthing{E}} \quad 
		\text{and} \quad 
		\left(\sell{\upthing{E}}\right)^{\circ}= \shf{\upthing{E}} .
	\end{equation}
\end{lemma}
Before we prove this lemma, let us mention several corollaries.
First, it follows that for any $s \in [0,\infty]$,
$\psi_s$ is indeed a strictly increasing admissible function,
and $s$-ellipsoidal functions are proper log-concave functions.
Moreover, as the support of $\shf{\ball{d+1}}$ is bounded for any $s \in [0, \infty),$
we conclude that $\psi_s$ is of linear growth for any $s \in [0, \infty).$
Also, we will need the following simple observations.
%	\begin{equation}%\label{eq:ell_monotonic}
%		\sell{\ball{d+1}}(x_1)\geq
%		\sell{\ball{d+1}}(x_2)\quad\text{for any}\ x_1,x_2\in\R^d\ \text{with}\ |x_2|\geqslant|x_1|.
%	\end{equation}
For any $\upthing{E} \in \ellipsd,$ we have
	\begin{equation}\label{eq:mahler_volume_alike}
		\int_{\Red} \shf{\upthing{E}} \cdot 
		\int_{\Red} \sell{\upthing{E}} =
		\int_{\Red} \shf{\ball{d+1}} \cdot 
		\int_{\Red} \sell{\ball{d+1}}.
	\end{equation}
Clearly, $\shf{\ball{d+1}}(x)$ is a strictly decreasing function of $s$
on $[0,+\infty)$ for any fixed  $x \in \ball{d} \setminus \{0\}.$ 
By this and~\eqref{eq:polarity_reverse}, we have that
	\begin{equation}\label{eq:ell_monotonic_in_s}
		\sell{\ball{d+1}}(x)\quad \text{is a strictly increasing function of}
		\ s\ \text{on}\ [0,+\infty)
		\ \text{for any fixed}
		\ x\in\R^d \setminus \{0\}.
	\end{equation}

\begin{proof}[Proof of Lemma~\ref{lem:duality_h_l}]
	Identities~\eqref{eq:duality_l_j} are trivial in two  cases $s=0$ and $s= \infty.$
	Assume $s \in (0,\infty).$
	By definition, for any log-concave function $f$ and any positive-definite operator $A$,
	we have
	$
		\loglego{(f\circ A)}=\loglego{f}\circ A^{-1}.
	$
	Thus, it suffices to consider $\upthing{E} = \ball{d+1}.$
	Consider $f = \left(\shf{\ball{d+1}}\right)^{2/s} = 1 - |x|^2,$
	then
	\[
		\loglego{f}(y)
		= \inf_{\enorm{x} < 1} \frac{e^{-\iprod{x}{y}}}{1 - \enorm{x}^2}
		= \inf_{t \in [0, 1)} \frac{e^{-t |y|}}{1 - t^2}
		= \frac{1 + \sqrt{1 + |y|^2}}{2} \exp\left(1 - \sqrt{1 + |y|^2}\right).
	\]

	Again, by the definition of the log-conjugate function, we have
	$
		\loglego{\left(f^q\right)}\!\left(x\right) 
		= \left[\loglego{f}\!\left(x/q\right)\right]^{{q}}
	$
	for all $x \in \R^d$ and $  q > 0.$
	Thus, we prove the leftmost identity in~\eqref{eq:duality_l_j},
	which implies the rightmost identity. 
\end{proof}

\subsection{L\"owner \tpdfs-functions}

We proved all the required results to understand
the properties of existence and uniqueness of a solution to problem~\eqref{eq:s_problem}.
\begin{proof}[Proof of Theorem~\ref{thm:s_lowner}]
	The result follows from Corollary~\ref{cor:psi_lin_lowner},
	Theorem~\ref{thm:lowner_uniqueness_general} and the fact that $\psi_s$
	is a strictly increasing admissible function of linear growth for a fixed $s \in [0, \infty).$
\end{proof}

The following lemma says that for fixed $s_1$ and $s_2,$
the integrals of the L\"owner functions
$\soelldense[s_1]{f}$ and $\soelldense[s_2]{f}$ are bounded by each other. 

Note that by~\eqref{eq:ell_monotonic_in_s},
$\ovolbs$ is an increasing function of $s \in [0,\infty).$
By simple computation, we have $\ovolbs[0] = d! \vol{d} \ball{d}.$
Actually, $\ovolbs$ can be computed using hypergeometric functions.
We claim without proof that 
	\begin{equation*}
		\ovolbs= \pi^{d/2} s^{d}
			\left(
				d \cdot 
					U\!\left(\frac{d}{2} + 1; d + \frac{s}{2} + 1; s\right)
					+ U\!\left(\frac{d}{2}; d + \frac{s}{2} + 1; s\right)
			\right),
	\end{equation*}
where $U(a;b;z)$ is the \emph{hypergeometric Tricomi} function defined by
	\begin{equation*}
		U(a;b;z) = 
			\frac{1}{\Gamma(a)} 
			\int_{0}^{+\infty} v^{a - 1}(v + 1)^{b - a - 1} e^{-z v}\di  v.
	\end{equation*}
We do not use this representation of $\ovolbs$ in our proofs.

\begin{lemma}\label{lem:integral_comparison}
	Let $f\colon \R^d \to [0, \infty)$ be a proper log-concave function  and $0\leq s_1<s_2$. Then
		\begin{equation}\label{eq:integral_comparison}
			\frac{\ovolbs[s_1]}{\ovolbs[s_2]}
			\leq \frac{\int_{\R^d} \soelldense[s_1]{f}}
				{\int_{\R^d} \soelldense[s_2]{f}}
			\leq
				\sqrt{\left(\frac{d+s_2}{s_2}\right)^{s_2}\left(\frac{d+s_2}{d}\right)^d}
			\cdot	
				\frac{\ovolbs[s_1]}{\ovolbs[s_2]}.
		\end{equation}
\end{lemma}
\begin{proof}
	We prove the leftmost inequality in~\eqref{eq:integral_comparison} first.
	Without loss of generality, assume that $\sell[s_1]{\ball{d+1}}$ is the L\"owner $s_1$-function of $f$.
	By~\eqref{eq:ell_monotonic_in_s}, we have that $f  \leq \sell[s_2]{\ball{d+1}}.$ 
	Hence, $\int_{\R^d} \soelldense[s_2]{f} \leq \ovolbs[s_2].$
	The leftmost inequality in~\eqref{eq:integral_comparison} follows.

	Now, we prove  the rightmost inequality in~\eqref{eq:integral_comparison}.
	Assume that $\sell[s_2]{\ball{d+1}}$ is the L\"owner $s_2$-function of $f$.
	For a fixed $\rho \in (0,1)$, consider
	\[
		\upthing{E}_{\rho}=
		\left(\rho \cdot \id \oplus \left({\sqrt{1-\rho^2}}\right)^{- s_2},0\right).
	\]
	We claim that $ \sell[s_2]{\ball{d+1}} \leq \sell[s_1]{\upthing{E}_{\rho}}$ for all $\rho\in(0,1).$ 
	By~\eqref{eq:polarity_reverse} and by~\eqref{eq:duality_l_j}, it is equivalent to the inequality
		$ \shf[s_1]{\upthing{E}_{\rho}} \leq \shf[s_2]{\ball{d+1}}.$
	Since for any $\rho\in(0,1)$ the cylinder 
	$\rho\ball{d} \times \left[0,\sqrt{1-\rho^2}\right]$ is contained in $\ball{d+1},$ 
	we conclude that 
	\[
		\shf[s_1]{\upthing{E}_{\rho}} \leq 
			\left(\sqrt{1 - \rho^2}\right)^{s_2} \cdot  \chi_{\rho\ball{d}} =
			\left(\sqrt{1 - \rho^2} \cdot  \chi_{\rho\ball{d}} \right)^{s_2}
			\leq \left(\shf[1]{\ball{d+1}}\right)^{s_2} =
		\shf[s_2]{\ball{d+1}}.
	\]
	Thus, $\sell[s_2]{\ball{d+1}} \leq \sell[s_1]{\upthing{E}_{\rho}}$ for all $\rho\in(0,1).$
	Therefore, $f \leq \sell[s_1]{\upthing{E}_{\rho}}$
	and $\int_{\R^d} \soelldense[s_1]{f} \leq \int_{\R^d} \sell[s_1]{\upthing{E}_{\rho}}$.
	Using~\eqref{eq:s_outer_volume} with $\rho = \sqrt{\frac{d}{d+s_2}},$ we obtain
	\[
		\frac{\int_{\R^d} \soelldense[s_1]{f}}
			{\int_{\R^d} \soelldense[s_2]{f}}
		= \frac{\int_{\R^d} \soelldense[s_1]{f}}
			{\int_{\R^d} \sell[s_2]{\ball{d+1}}} 	
		\leq \frac{\int_{\R^d} \sell[s_1]{\upthing{E}_{\rho}}}
				{\int_{\R^d} \sell[s_2]{\ball{d+1}}}
		= \sqrt{\left(\frac{d+s_2}{s_2}\right)^{s_2}\left(\frac{d+s_2}{d}\right)^d}
			\cdot	
				\frac{\ovolbs[s_1]}{\ovolbs[s_2]}.
	\]
\end{proof}

\subsection{The limit as \texorpdfstring{$s \to 0$ }{s --> 0}}

In this subsection, we prove Theorem~\ref{thm:zero_limit}.

\begin{proof}[Proof of Theorem~\ref{thm:zero_limit}]
	Let the L\"owner $s$-function $\soelldense{f}$ be represented by
	$(A_s\oplus\alpha_s, a_s)\in\mathcal{E}$ for every $s\in [0,1]$.
	Clearly, it suffices to show that 
		\begin{equation}\label{eq:lim_s_0}
			\lim_{s \to 0+} (A_s\oplus\alpha_s, a_s)
			= (A_0\oplus\alpha_0, a_0).
		\end{equation}
	Assume the contrary. 
	By Lemma~\ref{lem:integral_comparison},
	we see that the integrals of $\soelldense{f}$ for $s \in [0,1]$
	are bounded from above by some finite constant. 
	By monotonicity, $\left\{\ovolbs\right\}_{s\in[0,1]} \subset \left[\ovolbs[0], \ovolbs[1]\right]$ for any $d\in\N$. 
	Thus, applying Lemma~\ref{lem:psi_boundedness} for 
	$\psi= \psi_s,$ $s\in [0,1],$ we get that  the set 
	$\left\{(A_s\oplus\alpha_s, a_s)\right\}_{s \in [0,1]}$ is bounded in $\ellipsd.$ 
	Now, we see that there exists a sequence of positive numbers $\{s_i\}_1^{\infty}$
	with $\lim\limits_{i \to \infty} s_i = 0$ 
	such that
	\[
		\lim_{i \to \infty} (A_{s_i} \oplus \alpha_{s_i}, a_{s_i})
		= (A \oplus \alpha, a) \in \ellipsd 
		\quad \text{and}\quad  
		(A \oplus \alpha, a) \neq (A_0 \oplus \alpha_0, a_0).
	\]
	Clearly, we have that 
	\begin{equation}\label{eq:f_limit_comp}
		f \leq \sell[0]{(A\oplus\alpha, a)} \quad \text{and}
		\quad
		\int_{\R^d} \sell[0]{{(A \oplus \alpha, a)}}
		= \lim_{ i \to \infty}\int_{\R^d} \soelldense[s_i]{f}.
	\end{equation}
	By~\eqref{eq:ell_monotonic_in_s}, 
	we have $f \leq \sell{(A_0\oplus\alpha_0, a_0)}$ for any positive $s.$
	This and~\eqref{eq:f_limit_comp} imply that 
	\[
		\int_{\R^d} \sell[0]{( A \oplus \alpha, a)}
		= \lim\limits_{ i \to \infty}\int_{\R^d} \soelldense[s_i]{f}
		\leq \lim\limits_{ i \to  \infty} 
		\int_{\R^d}\sell[s_i]{(A_0 \oplus \alpha_0, a_0)}
		=\int_{\R^d} \soelldense[0]{f}.
	\]
	This contradicts the choice of $\soelldense[0]{f}.$
\end{proof}

\section{Gaussian densities. The limit as \texorpdfstring{$s \to \infty$ }{s --> infty}}\label{sec:gaussians}

In this Section, we prove Theorem~\ref{thm:lowner_infinity_existence}.
That is, we show that the limit of $\psi_s$-function as $s$ tends to $\infty$
may be only the gaussian distribution, which motivates our choice of the function $\psi_{\infty}$.

\subsection{Basic identities}

For any positive-definite operator $A,$ we have
\begin{equation}\label{eq:polar_gaussian}
	\loglego{\left(e^{-\enorm{A^{-1} x}^2}\right)} = 
	e^{-\frac{1}{4}\enorm{{A x}}^2}.
\end{equation}
By direct computations, we find  that
\begin{equation}\label{eq:limit_psi_at_infinity}
	\lim\limits_{s \to \infty} 
	\psi_s\!\left(\sqrt{{s}} \cdot t \right) = \frac{t^2}{2} 
	\quad \text{for any}\ t \in [0, + \infty).
\end{equation}
This implies that the functions 
$\sell{(c(s)I \oplus 1,0)},$ where $c(s) = \sqrt{s},$ converge to 
$e^{-{\enorm{x}^2}/2}$ as $s \to \infty.$
By this and using   assertion~\eqref{ass:lem_conv_dual_convergence_integrals}
of Lemma~\ref{lem:convergence_dual}, we have
\begin{equation}\label{eq:osvol_limit_s_infinity}
	\lim\limits_{s \to \infty}    \ovolbs \cdot s^{-d/2}= (2\pi)^{d/2}.
\end{equation}

As an immediate consequence of Lemma~\ref{lem:psi_boundedness} and 
identity~\eqref{eq:osvol_limit_s_infinity}, we obtain
\begin{cor}\label{cor:boundness_in_infinity}
	Let $f\colon\Red\to[0,\infty)$ be a proper log-concave function such that
	there is a sequence $\{s_i\}_1^{\infty}$ of positive numbers with 
	$\lim\limits_{i \to \infty} s_i = \infty$ and
	\begin{equation*}%\label{eq:cor_boundness_in_infinity_limit_in_ass}
		\lim\limits_{ i \to \infty}\int_{\R^d} \soelldense[s_i]{f} = \lambda < \infty.
	\end{equation*}
	Further, let $(A_s \oplus \alpha_s, a_s)$
	represent the L\"owner $s$-function of $f,$ $s \in [0, \infty).$
	Then inequality 
	\[
		\rho_1 \cdot \id  \prec \frac{A_{s_i}}{\sqrt{s_i}} \prec \rho_2 \cdot \id. 
	\]
	holds for some positive $\rho_1$, $\rho_2$ and all $i\in\N$.
\end{cor}

Using~\eqref{eq:duality_l_j} in assertion~\eqref{ass:lem_conv_dual_convergence_itself} of Lemma~\ref{lem:convergence_dual},
we conclude that
$\shf{\left(\sqrt{s} \cdot \id \oplus 1, 0 \right)}(x) \to e^{-\enorm{x}^2/2}.$
Again, by assertion~\eqref{ass:lem_conv_dual_convergence_integrals} of Lemma~\ref{lem:convergence_dual}, we get
\begin{equation*}
	\lim\limits_{s \to \infty} s^{d/2}   \int_{\R^d} \shf{\ball{d+1}} =   (2\pi)^{d/2}.
\end{equation*}
Thus, we conclude that
\begin{equation}\label{eq:osvol_volbs_limit}
	\lim\limits_{s \to \infty}    \ovolbs \cdot \int_{\Red} \shf{\ball{d+1}} =   (2\pi)^{d}.
\end{equation}
For any $(A \oplus \alpha, a) \in \ellipsd,$ identity~\eqref{eq:s_outer_volume} gives
\begin{equation}\label{eq:outer gaussian_volume}
	\int_{\Red} \sell[\infty]{(A \oplus \alpha, a)}  
	= \alpha \frac{\pi^{d/2}}{\det A}.
\end{equation}

\subsection{Limits of centered sequences}

\begin{lemma}\label{lem:limit_lowner_s-ell_infinity}
	Let $\{s_i\}_1^{\infty}$ be a sequence of positive scalars such that 
	$\lim\limits_{i \to \infty} s_i = \infty $,
	let $\{A_i\}_1^{\infty}$ be a sequence of positive-definite operators
	with
	$\lim\limits_{i \to \infty} \frac{A_i}{\norm{A_i}} = A,$
	where $A$ is positive-definite,
	and let the ellipsoids $\upthing{E}_{i},$
	represented by $(A_i \oplus 1, 0)$
	satisfy  
	$\lim\limits_{i \to \infty} \int_{\R^d} \sell[s_i]{\upthing{E}_i} =  \lambda < \infty.$ 
	Then the functions $\left\{\sell[s_i]{\upthing{E}_i}\right\}$
	converge uniformly on $\R^d$ to the Gaussian density $\sell[\infty]{(A_L \oplus 1, 0)},$ where
	\begin{equation}\label{eq:operator_inf_coef}
		A_L = \frac{\sqrt{\pi}}{\left(\lambda \det A \right)^{1/d}} A.
	\end{equation}
\end{lemma}
\begin{proof}
	Consider the functions $\left\{\shf[s_i]{\upthing{E}_i}\right\}_{1}^{\infty}.$ 
	By~\eqref{eq:osvol_volbs_limit}, we have that 
	\[
		\lim\limits_{i \to \infty}  \int_{\R^d} \shf[s_i]{\upthing{E}_i} = 
		\frac{(2 \pi)^d}{\lambda}.
	\]
	As shown in~\cite[Lemma 8.5]{ivanov2020functional},
	the functions $\shf[s_i]{\upthing{E}_i}$
	converge uniformly on $\R^d$ to the Gaussian density~$e^{-\enorm{A_J^{-1}x}^2}$
	with 
	\[
		A_J = \frac{2\sqrt{\pi}}{(\lambda \det A)^{1/d}} A. 
	\]
	Thus, by~\eqref{eq:polar_gaussian} and by assertion~\eqref{ass:lem_conv_dual_convergence_itself} of Lemma~\ref{lem:convergence_dual},
	we get that the functions $\sell[s_i]{\upthing{E}_i}$
	converge locally uniformly to the Gaussian density $\sell[\infty]{(A_L\oplus 1,0)}$ with $A_L$ given by~\eqref{eq:operator_inf_coef}. 
	Since \[
	\lim\limits_{i \to \infty}  \int_{\R^d} \sell[s_i]{\upthing{E}_i} = 
	\int_{\R^d} \sell[\infty]{(A_L\oplus 1,0)},\]
	we  conclude that 
	the functions $\left\{\sell[s_i]{\upthing{E}_i}\right\}$ are uniformly convergent on $\R^d$.
\end{proof}
\begin{lemma}\label{lem:lowner_gaussian_approx}
	Let $G$ be a Gaussian density. Then
	the functions $\soelldense{G}$ converge uniformly on $\R^d$ to $G$ as 
	$s \to \infty.$
\end{lemma}
\begin{proof}[Proof of Lemma~\ref{lem:lowner_gaussian_approx}]
	We assume that
	$G(x) = e^{-\enorm{x}^2/2}.$
	First, we relax the condition and prove
	that it suffices to approximate $G$ by any suitable $s$-ellipsoidal functions.
	\begin{claim}\label{claim:convenient_approx_gauss_lowner}
		If there is a function $c\colon [1, \infty) \to [1, \infty)$  such that  
		\begin{equation}\label{eq:convergence_l_standard_gaussian}
			\lim\limits_{s \to \infty}  \int_{\R^d} \sell{(c(s)I \oplus 1,0)} = 
			\int_{\R^d} G = \left(2 {\pi}\right)^{d/2},
		\end{equation}
		and  
		$G \leq \sell{(c(s)I \oplus 1,0)}$
		{for all}  $s \geq 1,$
		then the functions $\soelldense{G}$ converge uniformly on $\R^d$ to $G$ as 
		$s \to \infty.$
	\end{claim}
	\begin{proof}
		By Theorem~\ref{thm:s_lowner},
		the L\"owner function $\soelldense{G}$ exists and is unique for any positive $s$.
		By symmetry, we see
		that $\soelldense{G}$ is of the form 
		$\sell{\left(\beta(s)\id \oplus \alpha(s), 0 \right)},$
		where $\beta\colon [1, \infty) \to (0, \infty)$ and 
		$\alpha\colon [1, \infty) \to [1, \infty).$ By~\eqref{eq:convergence_l_standard_gaussian}, we obtain that
		\[
			\lim\limits_{s \to \infty}  \int_{\R^d} \soelldense{G} = \int_{\R^d} G.
		\]
		This implies that $\alpha(s) \to 1$ as $s \to \infty.$
		Hence, the functions
		$\soelldense{G}= \sell{\left(\beta(s)\id \oplus \alpha(s), 0 \right)}$
		converge uniformly on $\R^d$
		to the same function as the functions $\sell{\left(\beta(s)\id \oplus 1, 0 \right)}$
		as $s$ tends to $\infty$ (if the latter converges). 
		However, by Lemma~\ref{lem:limit_lowner_s-ell_infinity},
		the functions 
		$\sell{\left(\beta(s)\id \oplus 1, 0 \right)}$ 
		converge  uniformly on $\R^d$ to $G$ as $s \to \infty.$
	\end{proof}

	It is not hard to find a suitable function $c(s).$
	\begin{claim}\label{claim:explicit_approx_gauss_lowner}
		Let $c(s) = \sqrt{s}.$ Then
			$G \leq \sell{(c(s)I \oplus 1,0)}$
		{for all}  $s \geq 1$
		and identity~\eqref{eq:convergence_l_standard_gaussian} holds.
	\end{claim}
	\begin{proof}
		Identity~\eqref{eq:convergence_l_standard_gaussian} is an immediate consequence of~\eqref{eq:s_outer_volume} and~\eqref{eq:osvol_limit_s_infinity}.

		Inequality $G \leq \sell{(c(s)I \oplus 1,0)}$ is purely technical.  
		By~\eqref{eq:limit_psi_at_infinity}, for any $x \in \R^d,$ we have
		\[
			\lim\limits_{s \to \infty} \sell{(c(s)I \oplus 1,0)}(x) = G(x).
		\]
		We claim that
		$\sell{(c(s)I \oplus 1,0)}(x)$ 
		is a decreasing function of $s\in [1,\infty)$
		for any fixed $x \in \R^d.$  Or, equivalently, 
		that $\psi_s\! \left( \sqrt{s} \cdot t \right)$
		is an increasing function of $s\in [1,\infty)$
		for any fixed $t \in [0, \infty).$ 
		The derivative of  $\psi_s \left( \sqrt{s} \cdot t \right)$ as a function of $s$ is
		\[
			\frac{1}{2}\frac{t^2/s}{  \sqrt{1+ {t^2}/{s}}+ 1} - 
			\ln \left(\sqrt{1+\frac{t^2}{s}}+1\right) + \ln 2.
		\]
		It is a function of 
		$t^2/s,$ and we use $\Phi\!\left(t^2/s \right)$ to denote this function.
		Making the substitution $z = t^2/s$ and computing the derivative of 
		$\Phi(z)$,  we get that 
		\[
			\Phi'(z) = \frac{z}{4 \sqrt{z+1} \left(1 + \sqrt{z+1}\right)^2} > 0 .
		\]
		By this and since $\Phi(z) = 0,$ we conclude that 
		the derivative of  $\psi_s\!\left( \sqrt{s} \cdot t \right)$
		as a function of $s$ is positive for any $s \in [1, \infty).$ This completes the proof of the Claim.
	\end{proof}
	Lemma~\ref{lem:lowner_gaussian_approx} follows
	from Claims~\ref{claim:convenient_approx_gauss_lowner} and~\ref{claim:explicit_approx_gauss_lowner}.
\end{proof}

\subsection{Proof of Theorem~\ref{thm:lowner_infinity_existence}}

Implication \eqref{ass:thm:lowner_infinity_exists_unique}~$\Rightarrow$~\eqref{ass:thm:lowner_infinity_exists_any}
is trivial.
Implication \eqref{ass:thm:lowner_infinity_exists_any}~$\Rightarrow$~\eqref{ass:thm:lowner_infinity_liminf}
is a direct consequence of Lemma~\ref{lem:lowner_gaussian_approx}.

We proceed with implication \eqref{ass:thm:lowner_infinity_liminf}~$\Rightarrow$~\eqref{ass:thm:lowner_infinity_exists_unique}.
Let the L\"owner $s$-function $\soelldense{f}$
be represented by the ellipsoid $(A_s\oplus\alpha_s, a_s)\in\mathcal{E}$ for any $s \in [1, \infty).$
Applying Lemma~\ref{lem:psi_boundedness} with $\psi = \psi_s,$
one sees that there exists a minimizing sequence $\{s_i\}_1^{\infty}$ with 
$\lim\limits_{i \to \infty} s_i = \infty$
such that 
\begin{equation*}
%\label{eq:limits_s_infty_lowner}
	\lim\limits_{ i \to \infty} \int_{\R^d} \soelldense[s_i]{f}
	\to
	\liminf\limits_{s \to \infty} \int_{\R^d} \soelldense{f},
	\quad
	\frac{A_{s_i}}{\norm{A_{s_i}}} \to  A ,
	\quad
	\alpha_{s_i} \to \alpha 
	\quad \text{and} \quad
	a_{s_i}  \to a
\end{equation*}
for some positive semidefinite matrix $A\in\Re^{d\times d}$, a number 
$\alpha>0$ and $a\in\Red$
as $i$ tends to~$\infty.$
By Corollary~\ref{cor:boundness_in_infinity},
we conclude that $A$ is positive-definite.
Thus, by Lemma~\ref{lem:limit_lowner_s-ell_infinity},
we have that the functions $\soelldense[s_i]{f}$
converge uniformly on $\R^d$ to some Gaussian density $G.$
Clearly, $ f \leq G$ and $\int_{\R^d} G = \liminf\limits_{s \to \infty} \int_{\R^d} \soelldense{f}.$
Hence, by Theorem~\ref{thm:psi_lowner},
there exists a unique solution to problem~\eqref{eq:s_infinity_problem}.
We use $G_L$ to denote this solution.
By Lemma~\ref{lem:lowner_gaussian_approx}, we have
\[
	\int_{\R^d} G_L = \liminf\limits_{s \to \infty}  \int_{\R^d} \soelldense{G_L} 
	\geq \liminf\limits_{s \to \infty}  \int_{\R^d} \soelldense{f}  =
	\int_{\R^d} G.
\]
Thus, by the choice of $G_L,$ we conclude that $G_L = G.$ 
Again, by Lemma~\ref{lem:lowner_gaussian_approx}, we have that
\[
	\lim\limits_{s \to \infty} \int_{\R^d} \soelldense{f} =
	\liminf\limits_{s \to \infty} \int_{\R^d} \soelldense{f} = \int_{\R^d} G_L.
\]
Hence,
$ \frac{A_{s}}{\norm{A_{s}}} \to  A , \; \alpha_{s} \to \alpha $
and  $ a_{s} \to a $ as $s \to \infty.$
Indeed, otherwise using Lemma~\ref{lem:limit_lowner_s-ell_infinity},
we see
that there is another Gaussian density $G_2$ such that $f \leq G_2$ 
and $\int_{\R^d} G_2 = \int_{\R^d} G_L,$
which contradicts the choice of $G_L.$
Thus, we conclude that the functions $\soelldense{f}$ converge uniformly on $\R^d$ as $s \to \infty$ to $G_L$,
completing the proof of Theorem~\ref{thm:lowner_infinity_existence}.

\begin{prp}\label{prp:gaussian_volume_l}
	Let $K \subset \Red$ be a convex body containing the origing in its interior
	and let $\norm{\cdot}_{K}$ denote the gauge function of $K$,
	that is, $\norm{x}_{K}=\inf\{\lambda>0\st x\in\lambda K\}$.
	Let $A^{-1}\! \left(\ball{d} \right)$
	be the smallest volume origin-centered ellipsoid containing  $K,$
	where $A$ is a positive-definite matrix.
	Then the $\infty$-ellipsoidal function (Gaussian density)
	represented by $\left(A \oplus 1, 0 \right)$
	is the unique solution to~\eqref{eq:s_infinity_problem}
	with $f = e^{-\norm{x}_K^2}.$
\end{prp}
\begin{proof}
	Let $(A^\prime \oplus \alpha^\prime, a^\prime)\in\ellipsd$ be such that 
	$f  \leq \sell[\infty]{(A^\prime \oplus \alpha^\prime, a^\prime)}.$ 
	First, we show that $K \subset (A^\prime)^{-1}\ball{d}.$ Indeed, we have
	\[
		\iprod{A^\prime (x - a^\prime)}{A^\prime (x -a^\prime)} -
		\ln(\alpha^\prime)
		\leq \norm{x}_K^2 
	\]
	for every $x\in\Red$. Suppose for a contradiction that there is  
	$y \in K \setminus (A^\prime)^{-1} \ball{d}.$
	Consider $x=\vartheta y$, and 
	substitute into the previous inequality. We obtain 
	\[
		\vartheta^2 \enorm{A^\prime y}^2 - 
		2\vartheta
		\iprod{A^\prime a^\prime}{A^\prime y} + 
		\enorm{A^\prime a^\prime}^2 - 
		\ln(\alpha^\prime)
		\leq \vartheta^2\norm{y}_K^2 < \vartheta^2. 
	\]
	As $\enorm{A^\prime y} > 1,$ letting $\vartheta$ tend to infinity, 
	we obtain a contradiction. Thus, $ K \subset (A^\prime)^{-1}\ball{d}.$

	Hence, $\det A^\prime \leq \det A .$ On the other hand, 
	$\alpha^\prime \geq \norm{f}=1 .$
	 The proposition now easily follows from~\eqref{eq:outer gaussian_volume}.
\end{proof}

\section{Outer integral ratio}\label{sec:outer_integral_ratio}

The notion of the volume ratio 
was extended to the setting of log-concave functions in~\cite{alonso2018john}
and then in~\cite{ivanov2020functional} for John $s$-functions. 
For any $s \in [0, \infty),$ the \emph{$s$-integral ratio} of $f$
is defined by
\[
	\sintrat(f) = 
	\left( 
	\frac{\int_{\Red} f }{\int_{\Red} \selldense{f}}
	\right)^{1/d}.
\]
Corollary~1.3 of~\cite{alonso2018john} states
that there exists $\Theta > 0$
such that 
\[
	\sintrat[0](f) \leq \Theta \sqrt{d}
\]
for any proper log-concave function $f\colon \R^d \to [0,\infty)$
and any positive integer $d.$
As a simple consequence of this result,
the authors of~\cite{ivanov2020functional}
generalized this asymptotically tight bound to the John $s$-functions with fixed $s \in [0, \infty).$

However, the similar property of the L\"owner ellipsoid is not obtained in~\cite{li2019loewner}. 
For any $s \in [0, \infty),$
it is reasonable to define the \emph{outer $s$-integral ratio} of $f$ by
\[
	\sointrat(f) = 
	\left( 
	\frac{\int_{\Red} \soelldense{f} }{\int_{\Red}  f}
	\right)^{1/d}.
\]

Theorem~\ref{thm:s-outer_volume_ratio} is an immediate corollary of the following result and Lemma~\ref{lem:integral_comparison}.
\begin{thm}\label{thm:outer_integral_ratio}
	There exists $\Theta_0$
	such that for any positive integer $d$ and a proper log-concave function $f\colon \Red \to [0,\infty)$,
	the following inequality holds
	\[
		\sointrat[0](f) \leq  \Theta_0 
		\sqrt{d}.
	\]
\end{thm}
\begin{proof}[Proof of Theorem~\ref{sec:outer_integral_ratio}]
	Since the outer integral ratio is the same for all functions
	of the form  $\alpha f(x-a)$ with $\alpha > 0$ and $a \in \R^d,$
	we assume that  $f(0)= \norm{f}= e^{-d}.$
	By Lemma~\ref{lem:psi_lin_growth_bound}, 
	there is a $d$-ellipsoid $(A \oplus 1, 0)$ 
	such that $f \leq \sell[0]{(A \oplus 1,0)}.$
	By Lemma~\ref{lem:psi_boundedness} and Corollary~\ref{cor:existence},
	there exists a $0$-ellipsoidal function of minimal integral
	among those even $0$-ellipsoidal functions with height $1$
	that are pointwise greater than or equal to $f.$ 
	Applying a suitable linear transform,
	we assume that the minimal integral is attained at 
	$\sell[0]{\ball{d+1}} = e^{-\enorm{x}}.$ 
	That is, we assume that $ f \leq \sell[0]{\ball{d+1}}$ and $\sell[0]{\ball{d+1}}$
	is the solution to problem 
	\[
		\min\limits_{(A \oplus 1, 0) \in \ellipsd } 
		\int_{\R^d} \sell[0]{(A \oplus 1, 0)}  \quad \text{subject to} \quad f \leq \sell[0]{(A \oplus 1, 0)}.
	\]

	Let us show the geometric necessary condition for $\sell[0]{(A \oplus 1, 0)}$ to be the solution to this problem.
	Define the convex set 
	\[
		K_f=\operatorname{cl}
			\left(
				\bigcup_{r \geq d} \frac{1}{r} [f \geq e^{-r}]
			\right).
	\]

	\begin{claim}
		The unit ball $\ball{d}$
		is the smallest volume origin centered ellipsoid containing $K_f.$
	\end{claim}
	\begin{proof}
		By identity  $f(0) = \norm{f} = e^{-d}$ and  inequality $ f \leq \sell[0]{\ball{d+1}},$
		we see that each of the convex sets $\frac{1}{r} [f \geq e^{-r}]$ with $r \geq d$
		is nonempty and is a subset of the unit ball $\ball{d}.$
		Thus, $K_f \subset \ball{d}.$ 
		By construction, the origin belongs to $K_f.$

		It is easy to see that the smallest volume origin centered ellipsoid containing $K_f$ is unique.
		Assume, by a contradiction, that the ellipsoid 
		$A^{-1}\! \left(\ball{d}\right)$ with 
		$\det A > 1$ contains  $K_f.$ Consider the function
		$\sell[0]{(A \oplus 1,0)}.$ Fix $r \geq d.$ By definition, we have 
		\[ 
			\left[\sell[0]{(A \oplus 1,0)} \geq e^{-r}\right] = 
			\left\{x \in \R^{d} \st e^{-\enorm{Ax}} \geq e^{-r} \right\} = 
			\left\{x \in \R^{d} \st \enorm{Ax} \leq r \right\} =
			r \cdot A^{-1}\! \left(\ball{d}\right).
		\]
		This and 
		$\frac{1}{r} [f \geq e^{-r}] \subset A^{-1}\! \left(\ball{d}\right)$
		yield
		\[
			[f \geq e^{-r}] \subset \left[\sell[0]{(A \oplus 1,0)} \geq e^{-r}\right]. 
		\]
		Thus, $f \leq \sell[0]{(A \oplus 1,0)}$ and
		$\int_{\R^d} \sell[0]{(A \oplus 1,0)} < \int_{\R^d} \sell[0]{\ball{d+1}}.$
		We obtain a contradiction.
	\end{proof}

	Adjusting the celebrated John's theorem
	(for example, following the proof given in~\cite{ball1997elementary})
	to our case, we get the following statement,
	the  proof of which is given in Appendix~\ref{sec:appendix}.

	\begin{claim}\label{claim:lowner_for_centered_ellipsoid}
		Let $K$ be a convex body in $\R^d.$
		If $\ball{d}$ is the smallest volume origin symmetric ellipsoid containing $K,$
		then there exist unit vectors $(u_i)_1^m$ on the boundary of $K$ and positive weights $(c_i)_1^m$
		satisfying John's condition
			\begin{equation}\label{eq:john_condition}
				\sum_{i=1}^{m} c_i u_i \otimes u_i = \id  
				\quad \text{and} \quad 
				\sum\limits_{i=1}^m c_i=d.
			\end{equation}
	\end{claim}

	We use another key tool related to John's condition, namely, the Brascamp--Lieb inequality.
	We recall the reverse  Brascamp--Lieb inequality of F. Barthe~\cite[Theorem 5]{barthe1998reverse}. 
	It states that if unit vectors $(u_i)_1^m$ and positive weights $(c_i)_1^m$ satisfy John's condition~\eqref{eq:john_condition},
	then for measurable non-negative functions $q_i$ on $\R ,$ $i\in[m]$, one has
	\begin{equation}\label{eq:reverse_B-L_inequality}
		\int_{\R^d}\sup\left\{\prod_{i=1}^m q_i(\Theta_i)^{c_i} \st x
			= \sum_{i=1}^m c_i \Theta_i u_i\right\} \di x \geq \prod_{i=1}^m \left(\int_{\R} q_i\right)^{c_i}.
	\end{equation}
		
	Let unit vectors $(u_i)_1^m$ on the boundary of $K_f$ and positive weights $(c_i)_1^m$
	satisfy John's condition~\eqref{eq:john_condition}.	
	Fix $i \in [m].$
	Let $f_i$ be the restriction of $f$ on the line with directional vector $u_i,$ 
	that is, $f_i(t) = f(t u_i),$ $t \in \R.$	 	
	What does it mean that $u_i$ belongs to the boundary of $K_f$?
	The answer is simple: either there exists $t \geq d$ such that 
	$f_i (t) = e^{-t}$
	or, by convexity, $u_i$ is an accumulation point
	of the segments $\left\{\frac{1}{r} [f_i \geq e^{-r}]\right\}_{r \geq d}$ 
	and does not belong to any of these segments.
	In the former case, we set 
	$t_i = \min\{t \geq 0 \st f(tu_i) = e^{-t} \}.$
	Otherwise, put $t_i = + \infty.$
	Clearly,  $t_i \geq d$.
	Define for a finite $t_i$
	\[
		g_i(t)=
			\begin{cases}
				e^{-d} e^{-t \left(1 - \frac{d}{t_i}\right)}, \quad & \text{if} \quad t\in[0,t_i),\\
				0, \quad & \text{otherwise}
			\end{cases}
	\]
	and for $t_i = + \infty$
	\[
		g_i(t)=
			\begin{cases}
				e^{-d} e^{-t}, \quad & \text{if} \quad t\in[0, \infty),\\
				0, \quad & \text{otherwise}
			\end{cases}.
	\]
	\begin{claim}\label{claim:ovr_f_g_comparison}
		For any $i \in [m],$ we have that $g_i \leq f_i.$
	\end{claim}
	\begin{proof}
		If $t_i$ is finite, the result follows from the the log-concavity of $f_i.$
		Consider the case $t_i = + \infty.$
		Then for any $\epsilon \in (0,1),$ there exists 
		$r >  1 / \epsilon$ such that  $f_i((1- \epsilon)r) \geq e^{-r}.$
		The log-concavity of $f_i$ yields 
		\[
			f_i(t) \geq 
			e^{-d} 
			e^{-t \left[\frac{1}{1 - \epsilon}  - \frac{d}{(1- \epsilon)r}\right]}
		\]
		for all $t \in [0,1 / \epsilon].$
		Taking the limit as $\epsilon$ tends to $0,$ we see that
		$f_i (t) \geq e^{-d} e^{-t}$ for all $t \in [0, \infty).$
		This completes the proof of Claim~\ref{claim:ovr_f_g_comparison}.
	\end{proof}

	By the log-concavity of $f$ and the rightmost identity in~\eqref{eq:john_condition},
	we get
		\begin{equation}\label{eq:pre_BL}
			f(x) \geq \sup 
			\left\{ 
			\prod_{i=1}^m f_i(d \cdot \Theta_i )^{c_i / d} \st x=\sum_{i=1}^m c_i \Theta_i u_i \right\}.
		\end{equation}
	Using~\eqref{eq:pre_BL} in the reverse Brascamp--Lieb inequality~\eqref{eq:reverse_B-L_inequality}
	and by Claim~\ref{claim:ovr_f_g_comparison},
	we obtain 
	\[
		\int_{\R^d} f \geq \prod_{i=1}^m \left(\int_{\R} f_i(t d)^{1 / d} \di t \right)^{c_i} \geq
		\prod_{i=1}^m \left(\int_{\R} g_i(t d)^{1 / d} \di t \right)^{c_i}
			= \prod_{i=1}^m \left(\int_{0}^{t_i / d} g_i(t d)^{1 / d}  \di t\right)^{c_i}.
	\]
	%			\[
	%			= \prod_{i=1}^m \left(\frac{1}{e} \cdot \frac{1 - e^{1 - t_i / d}}{1 - \frac{d}{t_i}}\right)^{c_i}.
	%\]
	For any $\tau \geq d,$ define $p(\tau) = \int\limits_{0}^{\tau/d} e^{-t \left(1 - d / \tau\right)} \di t.$
	Additionally, put $p(+\infty) = e^{-d} \int\limits_{0}^{+\infty} e^{-t} \di t.$
	Therefore, by the previous inequality and the rightmost identity in~\eqref{eq:john_condition},
	we get
	\[
		\int_{\R^d} f \geq 
		\prod_{i=1}^m \left(\int_{0}^{t_i / d} g_i(t d)^{1 / d}  \di t\right)^{c_i}   =
		\prod_{i=1}^m \left( \frac{1}{e} \int_{0}^{t_i / d} 
		e^{-t \left(1 - d/t_i\right)} \di t \right)^{c_i}	\geq	
		e^{-d} \cdot  
		\left(\inf\limits_{\tau \in [d, + \infty]} p(\tau) \right)^{d}.
	\]
	We claim that $\inf\limits_{\tau \geq d} p(\tau) = 1.$
	Indeed, $p(d) = 1$ and $\lim\limits_{\tau \to \infty} p(\tau) = p(+\infty) = 1;$
	on the other hand, if $\tau\in[d,+\infty)$, we have
	\[
		p(\tau) =
			\frac{1 - e^{1 - \tau / d}}{1 - d / \tau} > 1
			\quad  \Longleftrightarrow \quad  e^{\tau/d} > \frac{e\tau}{d}.
	\] 
	The last inequality is simple and holds for any $\tau\in[d,+\infty)$.
	Thus, $	\int_{\R^d} f  \geq e^{-d},$ and  we conclude that
	\[
		\sointrat[0](f)
			=\left(\frac{\int_{\R^d} \soelldense[0]{f}}{\int_{\R^d} f}\right)^{1 / d}
			\leq \left(\frac{\int_{\R^d} e^{-\enorm{x}} \di x}{\int_{\R^d} f}\right)^{1 / d}
			\leq \left(\frac{d! \vol{d}\ball{d}}{e^{-d}}\right)^{1 / d}
			= e {\sqrt{\pi}} \left(\frac{d !}{\Gamma\!\left(1 + d/2\right)}\right)^{1 / d},
	\]
	where $\Gamma(\cdot)$ is Euler's Gamma function. The existence of $\Theta_0$ follows.
\end{proof}

\section{Duality}\label{sec:duality}

Recall that the polar set $K^{\circ}$ of a body $K \subset \R^d$
is defined by
\[
	K^{\circ} = \{y\in\R^d \st \forall x \in K\ \iprod{y}{x} \leq 1\}.
\]
It is known that for a convex body $K \subset \R^d$ with the John ellipsoid $J_K$ centered at the origin, we have  
\[
	(J_K)^\circ = L_{K^{\circ}},
\]
where $L_{K^{\circ}}$ is the L\"owner ellipsoid of  ${K^{\circ}}.$ 

One would expect the same properties for  John and L\"owner functions (see Subsections~\ref{subsec:john_s-func} and \ref{subsec:main_results}).
However, as was shown in~\cite{li2019loewner},
there is no duality between the John and L\"owner $0$-functions.
That is, assuming the John $0$-function of $f$
is represented by an ellipsoid centered at the origin, we might have 
$\loglego{\left(\selldense[0]{f}\right)} \neq \soelldense[0]{\loglego{f}}.$
By continuity and Theorem~\ref{thm:zero_limit},
it follows that there exist functions $f$ such that
$\loglego{\left(\selldense{f}\right)} \neq \soelldense[s]{\loglego{f}}$
for a sufficiently small positive $s.$
We conjecture that for any positive $s$ there is such an example.

The problem is that the centers of $d$-ellipsoids
representing $\selldense{f}$ and $\soelldense{\loglego{f}}$
might be different.
It is not the case for the setting of convex sets
(key observation:
the polar of an ellipsoid containing the origin in the interior is an ellipsoid).
However, for any $s \in [0, \infty],$ 
if $d$-ellipsoids representing $\selldense{f}$ and $\soelldense{\loglego{f}}$
are origin symmetric,
we have the duality, that is, the following identity 
	\begin{equation}\label{eq:dual_s-func_centered}
		\loglego{\left(\selldense{f}\right)}=\soelldense{f^\circ}
		\quad\text{and}\quad 
		\loglego{\left(\soelldense{f}\right)}=\selldense{f^\circ}
	\end{equation}
holds. Indeed, it follows from Lemma~\ref{lem:duality_h_l},
identity~\eqref{eq:mahler_volume_alike} and~\eqref{eq:polarity_reverse}.

Here are some examples of such a  duality.
By symmetry, we have the following.
\begin{exl}\label{exl:john-lowner-function}
	Let $s \in [0, \infty]$ and $f\colon \Red \to [0, \infty)$
	be a proper even log-concave function on $\Red.$
	Then identity~\eqref{eq:dual_s-func_centered} holds.
\end{exl}
In Theorem 5.1 of~\cite{ivanov2020functional},
the authors give a necessary and sufficient condition
for a proper log-concave functions $f$
to satisfy $\shf{\ball{d+1}} = \selldense{f}$, $s \in (0, \infty)$. 
This and Example~\ref{exl:john-lowner-function}
yield the same type result for proper even log-concave functions.
 
In the following example, the centers were computed directly.
By Proposition~\ref{prp:gaussian_volume_l} and Proposition 8.3 of~\cite{ivanov2020functional},
we obtain
\begin{exl}
	Let $K \subset \Red$ be a convex body containing the origin it its interior and let 
	$\norm{\cdot}_{K}$ denote the gauge function of $K$, that is, 
	$\norm{x}_{K}=\inf\{\lambda>0\st x\in\lambda K\}.$ Set  $f = e^{-\norm{x}_K^2}.$ Then 
	\begin{equation*}
		\loglego{\left(\selldense{f}\right)}=\soelldense{f^\circ}
		\quad\text{and}\quad 
		\loglego{\left(\soelldense{f}\right)}=\selldense{f^\circ}.
	\end{equation*}
\end{exl}

Finally, we note that the duality helps us to understand the properties of the constructed L\"owner functions.
The key tool for the interpolation between ellipsoidal functions $\shf{(\cdot)}$
is the \emph{Asplund sum} (or, \emph{sup-convolution}),
which is defined for  two functions $f_1$ and $f_2$ on $\Red$
by
\[
	(\loginfconv{f_1}{f_2})(x) = \sup\limits_{x_1 + x_2 = x} f_1(x_1) f_2(x_2).
\]
In our case everything is easier;
we use the product of two $\sell{(\cdot)}$ functions. 
It is easy to check that
\[
	\loglego{(\loginfconv{f_1}{f_2})} =
	\loglego{f_1} \cdot \loglego{f_2}.
\]
That is, despite the fact
that there is no duality between the John and L\"owner $s$-functions in general,
there is duality between the methods.

\appendix

\section{}\label{sec:appendix}

\begin{proof}[Proof of Claim~\ref{claim:lowner_for_centered_ellipsoid}]
We equip the space of symmetric matrices of order $d$
with an inner product defined by
\[ 
	\iprod{A}{B} = \tr{ A \, B}.
\]
Denote the set of contact points by $C=\bd{\ball{d}}\cap\bd{{K}}$
and consider
\[
	 \widehat{C}=\left\{{u}\otimes {u} \st 
	{u}\in C\right\}.
\]

The proof of Claim~\ref{claim:lowner_for_centered_ellipsoid} 
is an adaptation of the argument given in~\cite{ball1997elementary} and~\cite{GruberBook}. 

First, as a standard observation, we state the relationship between 
\eqref{eq:john_condition} and separation by a hyperplane of the point 
$\id$ from the set $\pos{\widehat{C}}$ in the space of symmetric matrices.

We claim that the following assertions are equivalent:
\begin{enumerate}
	\item\label{item:eqjohncond} 
		There are contact points and weights satisfying~\eqref{eq:john_condition}.
	\item\label{item:eqjohncond2} 
		There are contact points and weights satisfying the leftmost identity in~\eqref{eq:john_condition}.
	\item\label{item:sinpos}
		$\id \in \pos \widehat{C}$.
	% \item\label{item:sinconv} 
	% $ \frac{1}{d}\id \in\conv{\! \widehat{C}}$.
	\item\label{item:separation} 
		There is no symmetric matrix $H \in \R^{d \times d}$ with
\begin{equation}\label{eq:Hseparates}
	\iprod{{H}}{ \id}>0, \text{ and }
	\iprod{{H}}{{u} \otimes {u}} < 0 
	\text{ for all } {u}\in C.
\end{equation}
%\item\label{item:separationweaker}
%There is no $(\upthing{H},h)$ with
%\begin{equation}\label{eq:Hseparatesweak}
% \iprod{(\upthing{H},h)}{(\upthing{S},0)}>0, \text{ and }
% \iprod{(\upthing{H},h)}{(\upthing{u}\otimes\upthing{u},u)}\leq0 
% \text{ for all }\upthing{u}\in C.
%\end{equation}
\end{enumerate}
Indeed, taking the trace of both sides in the leftmost identity  in~\eqref{eq:john_condition}, we get that
$\sum\limits_{i=1}^m c_i=d.$ Thus, the 
equivalence of~\eqref{item:eqjohncond} and~\eqref{item:eqjohncond2} follows.
%, as well as that of 
%\eqref{item:sinpos} and ~\eqref{item:sinconv} follow.
The equivalence of~\eqref{item:eqjohncond2},~\eqref{item:sinpos} and~\eqref{item:separation}
is an immediate consequence of the hyperplane separation lemma.

Assume that there are no contact points and weights
satisfying the leftmost identity in~\eqref{eq:john_condition}.
Then there is a symmetric matrix $H \in \R^{d \times d}$
satisfying~\eqref{eq:Hseparates}. 

Consider the matrix $\id + \delta H.$
By~\eqref{eq:Hseparates}, 
for some $\delta_1 > 0$ and for any $\delta \in (0, \delta_1),$ we obtain 
\begin{equation*}
%\label{eq:matrix_det_bound}
\det \left( \id + \delta H\right) = 
1 + \delta \iprod{\id}{H} + o(\delta) > 1.
\end{equation*}
Hence, the ellipsoid 
$\left( \id + \delta H\right)^{-1} \left(\ball{d} \right)$
has the volume strictly less than the volume of $\ball{d}$ for any  $\delta \in (0, \delta_1).$ 

Thus, it  suffices to show that  that there exists $\delta_0 > 0$ such that
for any $\delta \in [0, \delta_0),$ the ellipsoid 
$\left( \id + \delta H\right)^{-1} \left(\ball{d} \right)$ contains $K.$ Let us prove this assertion.

Fix any $u \in C$ and denote $x = x(\delta, u) = \left( \id + \delta H\right) u.$ Then
\[
\enorm{x}^2 = \iprod{ \left( \id + \delta H\right) u}{ \left( \id + \delta H\right) u} = 1 + 2 \delta 
\iprod{u}{H u} + \delta^2 \enorm{Hu}^2.
\]
By~\eqref{eq:Hseparates}, we conclude that there exists
$\delta_u >  0$ such that $\enorm{x} < 1$ for all
$\delta \in (0, \delta_u).$ Hence, $u = \left( \id + \delta H\right)^{-1} x$ belongs to the interior of 
$\left( \id + \delta H\right)^{-1} \left(\ball{d} \right)$
for all $\delta \in (0, \delta_u).$
By compactness, there exists $\delta_2  > 0$ such that 
$C$ belongs to the interior  of $\left( \id + \delta H\right)^{-1} \left(\ball{d} \right)$
for all  $\delta \in (0, \delta_2).$
Again, by compactness, it follows that there exists $\delta_0 > 0$ such that 
for any $\delta \in [0, \delta_0),$ the ellipsoid 
$\left( \id + \delta H\right)^{-1} \left(\ball{d} \right)$ contains $K.$
That is, the original symmetric ellipsoid 
$\left( \id + \frac{\delta_0}{2} H\right)^{-1} \left(\ball{d} \right)$
contains $K$ and has the volume that is strictly less than 
the volume of $\ball{d}.$
We obtain a contradictionwith the choice of the smallest volume origin symmetric ellipsoid containing $K,$
and complete the proof of Claim~\ref{claim:lowner_for_centered_ellipsoid}.
\end{proof}
\begin{rem}
	It is not hard to prove that under the assertion of Claim~\ref{claim:lowner_for_centered_ellipsoid},
	the unit ball $\ball{d}$ is the smallest volume origin centered ellipsoid containing $K$
	if and only if John's condition~\eqref{eq:john_condition} holds.
\end{rem}

\subsection*{Acknowledgement}
The authors acknowledge the financial support from the Ministry of Educational and Science
of the Russian Federation in the framework of MegaGrant no 075-15-2019-1926

\bibliographystyle{alpha}
\bibliography{../../../../work_current/uvolit.bib}
\end{document}